\documentclass[11pt,leqno]{article}
\usepackage{amsmath, amscd, amsthm, amssymb, graphics, xypic, mathrsfs, setspace, fancyhdr, times}

\DeclareFontFamily{OT1}{pzc}{}
\DeclareFontShape{OT1}{pzc}{m}{it}%
             {<-> s * [1.195] pzcmi7t}{}
\DeclareMathAlphabet{\mathscr}{OT1}{pzc}%
                                 {m}{it}
\usepackage[pagebackref=true]{hyperref}
\hypersetup{backref}

\newcommand{\tensor}{\otimes}
\newcommand{\colim}{\operatorname{colim}}
\newcommand{\Spec}{\operatorname{Spec}}
\newcommand{\Aut}{{\mathbf{Aut}}}
\newcommand{\isomto}{{\stackrel{\sim}{\;\longrightarrow\;}}}
\newcommand{\isomt}{{\stackrel{{\scriptscriptstyle{\sim}}}{\;\rightarrow\;}}}

\renewcommand{\O}{{\mathcal O}}
\renewcommand{\hom}{\operatorname{Hom}}

\newcommand{\cplx}{{\mathbb C}}

\newcommand{\Z}{{\mathbb Z}}
\newcommand{\aone}{{\mathbb A}^1}
\newcommand{\pone}{{\mathbb P}^1}

\newcommand{\gm}{{{\mathbf G}_{m}}}
\renewcommand{\L}{{\mathcal L}}
\newcommand{\MW}{\mathrm{MW}}

\newcommand{\et}{\text{\'et}}
\newcommand{\ho}[1]{\mathscr{H}({#1})}
\newcommand{\hop}[1]{\mathscr{H}_{\bullet}({#1})}

\newcommand{\bpi}{\boldsymbol{\pi}}

\newcommand{\Nis}{\operatorname{Nis}}

\newcommand{\Sm}{\mathscr{Sm}}

\newcommand{\Spc}{\mathscr{Spc}}

\newcommand{\Ab}{\mathscr{Ab}}
\newcommand{\Gr}{\mathscr{Gr}}

\newcommand{\K}{{{\mathbf K}}}
\newcommand{\Pic}{\operatorname{Pic}}
\newcommand{\Faone}{{\mathbf{F}}_{\aone}}

\newcommand{\hsnis}{\mathscr{H}_s^{\Nis}(k)}
\newcommand{\hspnis}{\mathscr{H}_{s,\bullet}^{\Nis}(k)}

\newcommand{\F}{{\mathcal F}}

\newcounter{intro}
\theoremstyle{plain}
\newtheorem{thm}{Theorem}[section]

\newtheorem{lem}[thm]{Lemma}
\newtheorem{cor}[thm]{Corollary}
\newtheorem{prop}[thm]{Proposition}

\newtheorem*{thm*}{Theorem}
\newtheorem*{problem*}{Problem}

\newtheorem{thmintro}{Theorem}

\theoremstyle{definition}
\newtheorem{defn}[thm]{Definition}

\newtheorem{notation}[thm]{Notation}

\theoremstyle{remark}
\newtheorem{rem}[thm]{Remark}
\newtheorem{remintro}[thmintro]{Remark}

\newtheorem{ex}[thm]{Example}

\numberwithin{equation}{section}

\begin{document}
\pagestyle{fancy}
\renewcommand{\sectionmark}[1]{\markright{\thesection\ #1}}
\fancyhead{}
\fancyhead[LO,R]{\bfseries\footnotesize\thepage}
\fancyhead[LE]{\bfseries\footnotesize\rightmark}
\fancyhead[RO]{\bfseries\footnotesize\rightmark}
\chead[]{}
\cfoot[]{}
\setlength{\headheight}{1cm}

\author{\begin{small}Aravind Asok\thanks{Aravind Asok was partially supported by National Science Foundation Awards DMS-0900813 and DMS-0966589.}\end{small} \\ \begin{footnotesize}Department of Mathematics\end{footnotesize} \\ \begin{footnotesize}University of Southern California\end{footnotesize} \\ \begin{footnotesize}Los Angeles, CA 90089-2532 \end{footnotesize} \\ \begin{footnotesize}\url{asok@usc.edu}\end{footnotesize}}

\title{{\bf Splitting vector bundles \\ and $\aone$-fundamental groups of \\ higher dimensional varieties}}
\date{}
\maketitle

\begin{abstract}
We study aspects of the ${\mathbb A}^1$-homotopy classification problem in dimensions $\geq 3$ and, to this end, we investigate the problem of computing ${\mathbb A}^1$-homotopy groups of some ${\mathbb A}^1$-connected smooth varieties of dimension $\geq 3$.  Using these computations, we construct pairs of ${\mathbb A}^1$-connected smooth proper varieties all of whose ${\mathbb A}^1$-homotopy groups are abstractly isomorphic, yet which are not ${\mathbb A}^1$-weakly equivalent.  The examples come from pairs of Zariski locally trivial projective space bundles over projective spaces and are of the smallest possible dimension.

Projectivizations of vector bundles give rise to ${\mathbb A}^1$-fiber sequences, and when the base of the fibration is an ${\mathbb A}^1$-connected smooth variety, the associated long exact sequence of ${\mathbb A}^1$-homotopy groups can be analyzed in detail.  In the case of the projectivization of a rank $2$ vector bundle, the structure of the ${\mathbb A}^1$-fundamental group depends on the splitting behavior of the vector bundle via a certain obstruction class.  For projective bundles of vector bundles of rank $\geq 3$, the ${\mathbb A}^1$-fundamental group is insensitive to the splitting behavior of the vector bundle, but the structure of higher ${\mathbb A}^1$-homotopy groups is influenced by an appropriately defined higher obstruction class.
\end{abstract}

\begin{footnotesize}
\tableofcontents
\end{footnotesize}
\newpage
\section{Introduction}
The purpose of this paper is to study some aspects of the problem of classifying $\aone$-connected smooth proper varieties of a fixed dimension up to $\aone$-homotopy equivalence; this problem was introduced and discussed in \cite{AM}.  We focus here on varieties having dimension $\geq 3$ and, in particular, projective bundles over a smooth $\aone$-connected base.  For $\aone$-connected smooth varieties of dimension $\leq 2$ over an algebraically closed field having characteristic $0$, in \cite[\S 5]{AM} we showed the homotopy classification is largely governed by the $\aone$-fundamental sheaf of groups (see Section \ref{s:preliminaries} for a definition); henceforth we abbreviate $\aone$-fundamental sheaf of groups to $\aone$-fundamental group.

In dimensions $\leq 2$, it is possible to describe all $\aone$-connected smooth proper varieties over an algebraically closed field, and then to compute all the corresponding $\aone$-fundamental groups (again, see \cite[\S 5]{AM}).  In contrast, in dimension $\geq 3$, there seems to be no reasonable description of the collection of all $\aone$-connected smooth proper $3$-folds: by \cite[Corollary 2.3.7]{AM}, stably rational smooth proper complex $3$-folds are $\aone$-connected, and stably rational varieties are already quite mysterious (see also \cite[Example 2.3.4]{AM}).

This paper is motivated in part by the fact that it is still, in a sense, possible to provide a homotopy classification of all closed $3$-manifolds, even though we cannot answer the question of which finitely generated groups appear as fundamental groups of such manifolds \cite{Thomas, Swarup, Hendriks}; see Remark \ref{rem:3dclassification} for more discussion of this point.  We consider the question: what invariants are necessary for $\aone$-homotopy classification in dimensions greater than two?

Lens spaces provide examples of closed $3$-manifolds that are homotopy inequivalent yet whose homotopy groups are abstractly isomorphic (see Remark \ref{rem:lensspaces} for some quick recollections on lens spaces).  Therefore, if one believes that the analysis of $\aone$-homotopy theory of smooth proper $\aone$-connected $3$-folds is in any way similar to the analysis of homotopy theory of closed $3$-manifolds, then it seems reasonable to expect that knowledge of $\aone$-homotopy sheaves of groups (henceforth, $\aone$-homotopy groups) alone is not sufficient to distinguish all $\aone$-homotopy types in dimension $\geq 3$.  In support of this expectation, we establish the following result.

\begin{thmintro}[See Theorems \ref{thm:main}, \ref{thm:main2} and Remark \ref{rem:higherdimensionalexamplesI}]
\label{thmintro:homotopyinequivalentisomorphichomotopy}
Suppose $d$ is an integer $\geq 3$.  There exist pointed $\aone$-connected smooth proper $d$-folds $(X,x)$ and $(X',x')$ that fail to be $\aone$-weakly equivalent, yet for which $\bpi_i^{\aone}(X,x)$ is (abstractly) isomorphic to $\bpi_i^{\aone}(X',x')$ for all $i > 0$.
\end{thmintro}

The examples arise from arguably the next simplest class of varieties beyond projective spaces: projective bundles of direct sums of line bundles over projective spaces.  Voevodsky showed that integral motivic cohomology rings (in particular, Chow cohomology rings) are unstable $\aone$-homotopy invariants, and we distinguish the $\aone$-homotopy types of the examples in the previous theorem by direct computation of motivic cohomology rings.  The bulk of the work is devoted to studying $\aone$-homotopy groups of projectivizations of vector bundles over $\aone$-connected base schemes to establish the existence of the abstract isomorphism of the theorem statement.

By careful choice of the projective bundles, $\aone$-homotopy groups of degree $\geq 2$ can be controlled in a rather straightforward fashion.  On the other hand, the $\aone$-fundamental group of a positive dimensional smooth proper $\aone$-connected variety is always non-trivial \cite[Proposition 5.1.4]{AM}.  Very few computations of $\aone$-fundamental groups exist, and much of the paper is devoted to providing such computations for projective bundles over an $\aone$-connected smooth base.  Our computations can be viewed as providing a ``relative" version of Morel's fundamental computations \cite[\S 7.3]{MField} of ``low-degree" $\aone$-homotopy groups of ${\mathbb P}^n$; just as in topology, higher $\aone$-homotopy sheaves of ${\mathbb P}^n$ are isomorphic to $\aone$-homotopy sheaves of spheres and are therefore extremely complicated in general.

Before describing the computation, we mention another motivation for this investigation. An old question of Schwarzenberger asks about the existence of non-split small rank vector bundles on projective spaces (see, e.g., \cite[p. 227]{GIT} or \cite[\S 4.4]{OSS} for discussion of this question).  A precise form of this question is the conjecture of Hartshorne asserting that a rank $2$ vector bundle on ${\mathbb P}^n$ splits so long as $n \geq 7$; see \cite[Conjecture 6.3]{Hartshorne}, though no counterexamples are known even for $n = 6$.  Of course, if this conjecture is true, then the $\aone$-homotopy type of the projectivization of any rank $2$ vector bundle on projective space of large dimension is indistinguishable from the $\aone$-homotopy type of the projectivization of a split bundle.  At the very least, this circle of ideas suggests that the splitting behavior of a vector bundle influences the $\aone$-homotopy type, and this belief is borne out in our computations.

Just as in topology, we can study the $\aone$-homotopy groups of projective space bundles using the theory of $\aone$-fiber sequences, via techniques pioneered by F. Morel and extended by M. Wendt \cite{MField,Wendt}.  Indeed, given such a bundle, there is an associated long exact sequence relating the $\aone$-homotopy groups of the base, fiber and total space; see Corollary \ref{cor:projectivebundlelongexactsequence}.  Our main goals are thus twofold.  First, we want to provide explicit descriptions of the connecting homomorphisms; this is accomplished by study of the $\aone$-Postnikov tower (see Definition \ref{defn:aonepostnikov}).  Second, we want to solve the resulting extension problem to determine the group structure on the $\aone$-fundamental group. Proposition \ref{prop:projectivebundlesonaoneconnectedvars} shows that projective bundles on $\aone$-connected smooth varieties over algebraically closed fields of characteristic $0$ are automatically projectivizations of vector bundles, so discussion can be simplified by restricting to this case.

Proposition \ref{prop:splitvectorbundles} and Theorem \ref{thm:projectivebundletrivialeulerclass} demonstrate that if a vector bundle splits, then the $\aone$-homotopy groups of the associated projective bundle are very restricted.  In particular, the $\aone$-fundamental group of the projectivization of a split vector bundle is a split extension of the $\aone$-fundamental group of the base by the $\aone$-fundamental group of the fiber (i.e., projective space), and we completely describe the group structure on this extension.  Furthermore, the higher $\aone$-homotopy groups of a split bundle are simply a product of the $\aone$-homotopy groups of the base and fiber.

In contrast to the split case, a certain obstruction class, which we refer to as the Euler class (see Definition \ref{defn:eulerclass}), intercedes in the structure of the $\aone$-fundamental group of a potentially non-split bundle: the $\aone$-fundamental group of the total space is an extension of the $\aone$-fundamental group of the base by a quotient of the $\aone$-fundamental group of the fiber.  Moreover, this extension of sheaves of groups is no longer obviously split.  We show that, if our Euler class is trivial (which always happens, e.g., if the $\aone$-connected base variety is ${\mathbb P}^n$, $n \geq 2$) then i) the $\aone$-fundamental group of the total space is again a split extension of the $\aone$-fundamental group of the base by the $\aone$-fundamental group of the fiber, and ii) we can write down the group structure on the extension.

While $\bpi_1^{\aone}({\mathbb P}^n)$ is quite simple for $n \geq 2$, Morel showed \cite[\S 7.3]{MField}, the $\aone$-fundamental sheaf of groups of $\pone$ is a highly non-trivial (non-abelian!) sheaf of groups.  As a consequence, the study of $\aone$-fundamental groups of projectivizations of rank $2$ bundles is much more complicated than the case of higher rank bundles.  We recall Morel's description of $\bpi_1^{\aone}(\pone)$ in Theorem \ref{thm:centralextension}.  

The outlines of our approach are already visible in the study of homotopy groups of ${\mathbb{RP}}^n$-bundles over closed connected manifolds; we review this study as motivation at the beginning of \S \ref{s:splittingobstructions}.  The following statement summarizes our calculations in the situation where the $\aone$-connected base variety is a projective space (though more general results are established in the text).

\begin{thmintro}[See Theorems \ref{thm:rankn} and \ref{thm:projectivebundletrivialeulerclass}]
\label{thmintro:projectivebundles}
Suppose $\mathcal{E}$ is a rank $m$ vector bundle on ${\mathbb P}^n$, $n \geq 2$ and $m \geq 2$, and fix an identification $\Pic({\mathbb P}^n) \cong \Z$.  If $m \geq 3$, then there are isomorphisms
\[
\bpi_1^{\aone}({\mathbb P}({\mathcal E})) \isomto \gm \times \bpi_1^{\aone}({\mathbb P}^n).
\]
If $m = 2$, and furthermore $e(\mathcal{E})$ is trivial (see \textup{Definition \ref{defn:eulerclass}}), then there is a split short exact sequence of the form
\[
1 \longrightarrow \bpi_1^{\aone}(\pone) \longrightarrow \bpi_1^{\aone}({\mathbb P}(\mathcal E)) \longrightarrow \gm \longrightarrow 1;
\]
the class of the extension is determined by the image of $c_1(\mathcal{E})$ in $\Pic({\mathbb P}^n)/2\Pic({\mathbb P}^n) = \Z/2\Z$.
\end{thmintro}

\begin{remintro}
In \cite[Theorem 8.14]{MField}, Morel introduces an Euler class, defined via obstruction theory, that provides the only obstruction to splitting a rank $n$ vector bundle on a smooth affine scheme of dimension $\leq n$.  The Euler class we define is conceived in the same spirit: it provides an obstruction to splitting a vector bundle up to the $\aone$-$2$-(co)skeleton of the base (see Definition \ref{defn:aonepostnikov}) in a manner independent of the dimension of $X$ (or whether $X$ is affine), though at the expense of introducing some $\aone$-connectivity hypotheses.  It can be shown that our Euler class coincides with Morel's Euler class, but we have not provided a detailed verification here.  In a sense, our discussion is classical ({\em cf.} \cite[Chapter 12]{MilnorStasheff}), with some additional contortions necessitated in order to avoid imposing orientability hypotheses.
\end{remintro}

Regarding computations of $\aone$-homotopy groups, this work suggests two natural questions.  First, in the case where $\mathcal{E}$ is a rank $2$ vector bundle on an $\aone$-connected smooth variety and $e(\mathcal{E})$ is non-trivial, the theorem above is definitely not correct; we explain this in more detail in Example \ref{ex:nontrivialeulerclass}.  What does the $\aone$-fundamental group look like in this case?  Second, providing a description of higher $\aone$-homotopy groups of projective bundles will require more effort. Proposition \ref{prop:splitvectorbundles} provides some statements in the case of split vector bundles and Lemma \ref{lem:higherobstructions} gives a statement regarding possibly non-split vector bundles in the presence of strong $\aone$-connectivity hypotheses on the base variety (satisfied, for example, if the base variety is a high-dimensional projective space). Recent computations of higher $\aone$-homotopy sheaves of punctured affine spaces (see \cite{AsokFaselThreefolds,AsokFaselSpheres,AsokFaselpi3a3minus0}) open the door to providing a more detailed description of the next obstruction in the case of ${\mathbb P}^1$-bundles, and the first ``interesting" obstruction in the case of ${\mathbb P}^2$-bundles.

\subsubsection*{Overview of sections}
Section \ref{s:preliminaries} recalls a number of general facts from $\aone$-homotopy theory and $\aone$-algebraic topology, especially aspects related to the theory of the $\aone$-fundamental sheaf of groups.  The subject we consider is quite young an in an attempt to keep the paper reasonably self-contained, we have chosen to repeat statements of results.  Section \ref{s:fibersequences} collects a number of facts about specific $\aone$-fiber sequences related to classifying spaces for $SL_n, GL_n$ and $PGL_n$, which will be necessary for analyzing the various obstructions that arise.  Finally, the main results are proven in Sections \ref{s:splittingobstructions} and \ref{s:isomorphichomotopygroups}.  We refer the reader to the introduction to each section for a more detailed list of contents.

\subsubsection*{Acknowledgements}
We wish to thank Fabien Morel for many discussions related to the $\aone$-homotopy classification problem, and Brent Doran for discussions about the theory of variation of GIT quotients; taken together, these discussions provided the motivation for the construction of the examples required for Theorem \ref{thmintro:homotopyinequivalentisomorphichomotopy}.  We thank Marc Levine and Christian Haesemeyer for discussions about notions of motivic degree and Euler classes.  We thank Matthias Wendt for discussions regarding his work and a number of insightful comments on a draft of this paper.  We also thank Lukas-Fabian Moser for making preliminary versions of his results available to us, and Ben Williams for discussions about obstructions. Finally, we thank one of the referees of a previous version of this paper for providing an extremely detailed report with a number of comments and corrections that certainly improved its form and substance.

\section{Review of some $\aone$-algebraic topology}
\label{s:preliminaries}
In this section, we review some basic facts from $\aone$-homotopy theory and $\aone$-algebraic topology.  The material in this section will be used throughout the paper.  Much of what we write here can be found in the Morel's beautiful work \cite{MField}, but we have provided some proofs of results stated but not proven in {\em ibid}.

\subsubsection*{$\aone$-homotopic preliminaries}
Fix a field $k$; in the body of the text $k$ is assumed to have characteristic $0$ if no alternative hypotheses are given.  Write $\Sm_k$ for the category of schemes separated, smooth and finite type over $k$.  The word {\em sheaf} we will always mean Nisnevich sheaf on $\Sm_k$.  We write $\Spc_k$ for the category of simplicial Nisnevich sheaves on $\Sm_k$; objects of the category $\Spc_k$ will be referred to as $k$-spaces, or simply as {\em spaces} if $k$ is clear from context.  Sending a smooth scheme $X$ to the representable functor $\hom_{\Sm_k}(\cdot,X)$ and taking the associated constant simplicial sheaf (i.e., all face and degeneracy maps are the identity map) determines a fully faithful functor $\Sm_k \to \Spc_k$.  Spaces will generally be denoted by calligraphic letters, while schemes will be denoted by capital roman letters; we use the composite functor just mentioned to identify $\Sm_k$ with its essential image in $\Spc_k$ and we use the same notation for both a scheme and the associated space.  Write $\Spc_{k,\bullet}$ for the category of pointed $k$-spaces, i.e., pairs $({\mathcal X},x)$ consisting of a $k$-space ${\mathcal X}$ and a morphism $x: \Spec k \to {\mathcal X}$.

We write $\hsnis$ for the (Nisnevich) simplicial homotopy category and $\hspnis$ for its pointed analog; each of these categories is the homotopy category of a model structure on $\Spc_k$ or $\Spc_{k,\bullet}$.  For more details regarding the precise definitions, we refer the reader \cite[\S 2 Definition 1.2]{MV}.  The set of morphisms between two spaces in $\hsnis$ or $\hspnis$ will be denoted $[-,-]_s$, though when considering pointed homotopy classes of maps, the base-point will be explicitly specified unless it is clear from context.  For a pointed space $(\mathcal{X},x)$, we write $\Sigma^i_s$ for the simplicial suspension operation $S^i_s \wedge -$ and ${\bf R}\Omega^1_s \mathcal{X}$ for the space $\Omega^1_s \mathcal{X}^f$, where $\mathcal{X}^f$ is a simplicially fibrant model of $\mathcal{X}$.  As usual, simplicial suspension is left adjoint to (simplicial) looping.

We use the Morel-Voevodsky $\aone$-homotopy category $\ho{k}$ and its pointed analog $\hop{k}$; these categories are constructed as a (left) Bousfield localization of either $\hsnis$ or $\hspnis$; see \cite[\S 3 Definition 2.1]{MV}.  In particular, the category $\ho{k}$ is equivalent to the subcategory of $\hsnis$ consisting of $\aone$-local objects, and the inclusion of the subcategory of $\aone$-local objects admits a left adjoint called the $\aone$-localization functor.  There is a choice of such a functor, for which we write $L_{\aone}$, that commutes with formation of finite products (one model of the functor $L_{\aone}$ is given in \cite[\S 2 Lemma 3.20]{MV}, and one replaces $Ex$ in this lemma by the Godement resolution functor via \cite[\S 2 Theorem 1.66]{MV}).  We write $[-,-]_{\aone}$ for morphisms in $\ho{k}$ or $\hop{k}$ and explicitly indicate the base-point in the latter case.  

An important fact regarding the $\aone$-local model structure is that it is proper; in particular, pullbacks of $\aone$-weak equivalences along $\aone$-fibrations are again $\aone$-weak equivalences; see \cite[\S 2 Theorem 3.2]{MV}.  We will also import various definitions from classical homotopy theory (e.g., regarding connectivity) to the simplicial or $\aone$-homotopy category by prepending either the word simplicial or the symbol $\aone$; when we do not give precise definitions, the terms are defined in analogy with their classical topological counterparts.

If ${\mathcal X}$ is a space, the sheaf of simplicial connected components, denoted $\bpi_0^s(\mathcal{X})$, is the Nisnevich sheaf on $\Sm_k$ associated with the presheaf $U \mapsto [U,\mathcal{X}]_s$.  Similarly, the sheaf of $\aone$-connected components of ${\mathcal X}$, denoted $\bpi_0^{\aone}({\mathcal X})$ is the Nisnevich sheaf on $\Sm_k$ associated with the presheaf $U \mapsto [U,{\mathcal X}]_{\aone}$.  A space $\mathcal{X}$ is simplicially connected if the canonical morphism $\bpi_0^s(\mathcal{X}) \to \Spec k$ is an isomorphism, and $\aone$-connected if the canonical morphism $\bpi_0^{\aone}(\mathcal{X}) \to \Spec k$ is an isomorphism.  We recall the following result, which provides many examples of $\aone$-connected varieties, and suffices to establish $\aone$-connectedness of most of the examples in the remainder of the paper.

\begin{prop}
\label{prop:aoneconnectedness}
If $k$ is a field, and $X$ is a smooth scheme admitting an open cover by open sets isomorphic to affine space (with non-empty intersections), then $X$ is $\aone$-connected.
\end{prop}

If $({\mathcal X},x)$ is a pointed space, then $\bpi_i^s(\mathcal{X},x)$ is the Nisnevich sheaf on $\Sm_k$ associated with the presheaf $U \mapsto [S^i_s \wedge U_+,{\mathcal X}]_s$, and $\bpi_i^{\aone}(\mathcal{X},x)$ is the Nisnevich sheaf on $\Sm_k$ associated with the presheaf $U \mapsto [S^i_s \wedge U_+,{\mathcal X}]_{\aone}$ (here, we take pointed homotopy classes of maps).  These are sheaves of groups for $i > 0$, and sheaves of abelian groups for $i > 1$.  For notational compactness, we will often suppress base-points when writing $\aone$-homotopy groups when the space under consideration is either $\aone$-connected or if the choice of base-point is clear from context.

A presheaf $\F$ on $\Sm_k$ is $\aone$-invariant if for every $U \in \Sm_k$ the canonical map $\F(U) \to \F(U \times \aone)$ is a bijection.  A sheaf of groups $G$ (possibly non-abelian) is {\em strongly $\aone$-invariant} if the cohomology presheaves $H^i_{\Nis}(\cdot,G)$ are $\aone$-invariant for $i = 0,1$.  A sheaf of abelian groups $A$ is {\em strictly $\aone$-invariant} if the cohomology presheaves $H^i_{\Nis}(\cdot,A)$ are $\aone$-invariant for $i \geq 0$.  The main structural properties of the sheaves $\bpi_i^{\aone}(\mathcal{X},x)$ are summarized in the following fundamental results due to Morel.

\begin{thm}[{\cite[Theorem 6.1 and Corollary 6.2]{MField}}]
\label{thm:pi1stronglyaoneinvariant}
If $({\mathcal{X}},x)$ is a pointed space, then for any integer $i > 0$ the sheaves $\bpi_i^{\aone}(\mathcal{X},x)$ are strongly $\aone$-invariant.
\end{thm}

\begin{thm}[{\cite[Theorem 5.46]{MField}}]
\label{thm:stronglyaoneinvariantabelian}
If $k$ is a perfect field, a strongly $\aone$-invariant sheaf of abelian groups on $\Sm_k$ is strictly $\aone$-invariant.
\end{thm}

\begin{notation}
We write $\Gr^{\aone}_k$ for the category of strongly $\aone$-invariant sheaves of groups, and $\Ab^{\aone}_k$ for the category of strictly $\aone$-invariant sheaves of groups.
\end{notation}

\subsubsection*{Simplicial homotopy classification of torsors}
Let $G$ be a Nisnevich sheaf of groups.  Let $EG$ denote the \u Cech construction of the epimorphism $G \to \Spec k$, and let $BG$ denote the (Nisnevich) sheaf quotient $EG/G$ for the diagonal (right) action of $G$ on $EG$.  The space $BG$, which is the simplicial classifying space for $G$, has a canonical base point, which we denote $\ast$ in the sequel. We write $H^1_{\Nis}({\mathcal X},G)$ for the set of Nisnevich locally trivial $G$-torsors over ${\mathcal X}$, i.e., triples $(P,\pi,G)$ consisting of a (right) $G$-space $P$, a morphism $\pi: {\mathcal P} \to {\mathcal X}$ equivariant for the trivial right action on ${\mathcal X}$, and isomorphism of Nisnevich sheaves $G \times {\mathcal P} \isomt {\mathcal P} \times_{\mathcal X} {\mathcal P}$.  The terminology classifying space is justified by the following result.

\begin{thm}[{\cite[\S 4 Propositions 1.15 and 1.16]{MV}}]
\label{thm:simplicialhomotopyclassificationoftorsors}
If $G$ is a Nisnevich sheaf of groups, then for any space ${\mathcal X}$ there is a canonical bijection
\[
[{\mathcal X},BG]_{s} \isomto H^1_{\Nis}({\mathcal X},G).
\]
Moreover, for any integer $i > 0$, and any smooth scheme $U$, the group $[\Sigma^i_s U_+,BG]_s$ is isomorphic to $G(U)$ if $i = 1$ and is trivial if $i > 1$.
\end{thm}

If $G$ is a linear algebraic group viewed as a Nisnevich sheaf of groups and $X$ is a smooth scheme, the set $H^1_{\Nis}(X,G)$ studied above can actually be identified with the set of Nisnevich locally trivial $G$-torsors on $X$ in the usual sense.  For our purposes, it suffices to observe that a Nisnevich locally trivial $G$-torsor on $X$ in the usual sense gives rise to a $G$-torsor on $X$ in the sense above by means of the Yoneda embedding.

\begin{cor}
\label{cor:pi0linearalgebraicgroup}
If $G$ is a Nisnevich sheaf of groups, then $\bpi_0^{\aone}(BG) = \ast$.
\end{cor}

\begin{proof}
By the unstable $\aone$-connectivity theorem \cite[\S 3 Corollary 3.22]{MV} there is an epimorphism
\[
\bpi_0^s(BG) \longrightarrow \bpi_0^{\aone}(BG).
\]
The sheaf $\bpi_0^s(BG)$ is, by definition, the Nisnevich sheafification of $U \mapsto [U,BG]_{s}$.  To check triviality of this sheaf, it suffices to check triviality over stalks.  By Theorem \ref{thm:simplicialhomotopyclassificationoftorsors} and the discussion just prior to the statement, if $S$ is a Henselian local scheme, $[S,BG]_s$ corresponds to the set of isomorphism classes of Nisnevich locally trivial $G$-torsors over $S$.  However, Nisnevich locally trivial $G$-torsors over $S$ are trivial.
\end{proof}

\begin{rem}
\label{rem:colimits}
One immediate consequence of the above discussion is the fact that a Nisnevich sheaf of groups $G$ is strongly $\aone$-invariant if and only if the classifying space $BG$ is $\aone$-local; we use this observation repeatedly in the sequel.  Using Theorem \ref{thm:pi1stronglyaoneinvariant}, the inclusion of the subcategory of strongly $\aone$-invariant sheaves of groups into the category of Nisnevich sheaves of groups admits a left adjoint defined by $G \mapsto \bpi_1^{\aone}(BG)$: this left adjoint creates finite colimits, e.g., amalgamated sums.
\end{rem}

If $BG^f$ is a (simplicially) fibrant model of $BG$, then any element of $[{\mathcal X},BG]_s$ can be represented by a morphism of simplicial sheaves ${\mathcal X} \to BG^f$.  If ${\mathcal X}$ is simplicially connected, we can choose a base-point $x \in {\mathcal X}(k)$ making the aforementioned morphism a morphism of pointed simplicial sheaves.  Thus, any $G$-torsor on a simplicially connected space ${\mathcal X}$ can be represented by a pointed morphism ${\mathcal X} \to BG^f$ for an appropriate choice of base-point.  In the sequel, $X$ will be an $\aone$-connected smooth scheme, in which case the $\aone$-localization $L_{\aone}X$ is a simplicially connected space by the unstable $\aone$-connectivity theorem \cite[\S 3 Corollary 3.22]{MV}.

\begin{cor}
\label{cor:loopsclassifyingadjointness}
If $G$ is a Nisnevich sheaf of groups, there is a canonical simplicial weak equivalence ${\bf R}\Omega^1_s BG \isomt G$.
\end{cor}

\begin{proof}
For any smooth scheme $U$ there is an adjunction
\[
[\Sigma^i_s U_+,{\bf R}\Omega^1_s BG]_s \isomto [\Sigma^{i+1}_s U_+,BG]_s.
\]
By Theorem \ref{thm:simplicialhomotopyclassificationoftorsors}, if $i = 0$ the presheaf on the right hand side is precisely $G(U)$, and if $i > 0$ it is trivial.  In other words, after sheafification, the morphism of spaces ${\bf R}\Omega^1_s BG \to \bpi_0^s({\bf R}\Omega^1_s BG) = G$ is a simplicial weak equivalence.
\end{proof}

\subsubsection*{Generalities on fiber sequences}
We fix some results about $\aone$-fiber sequences; this material is taken from \cite[\S 6.2]{Hovey}.  In any pointed model category, there is a notion of loop object.  In the setting of $\aone$-homotopy theory, if $(\mathcal{X},x)$ is a pointed $\aone$-fibrant space, then the simplicial loop space $\Omega^1_s {\mathcal X}$ is precisely the model categorical notion of loop space.

If $p: {\mathcal E} \to {\mathcal B}$ is an $\aone$-fibration between pointed $\aone$-fibrant objects (equivalently, by \cite[\S 2 Proposition 2.28]{MV}, simplicially fibrant and $\aone$-local objects), let ${\mathcal F}$ be the fiber of this map, and $i: {\mathcal F} \to {\mathcal E}$ be the inclusion of the fiber.  The general formalism of model categories gives an action of $\Omega^1_s {\mathcal B}$ on ${\mathcal F}$, specified functorially.  Here is the construction.  Given a space ${\mathcal A}$, an element of $[{\mathcal A},\Omega^1_s {\mathcal B}]_{\aone}$ can be represented by a morphism $h: {\mathcal A} \times \Delta^1_s \to {\mathcal B}$ (where $\Delta^1_s$ is the simplicial interval).  If $u: {\mathcal A} \to {\mathcal F}$ is a morphism, by composition we get a morphism $u' = i \circ u: {\mathcal A} \to {\mathcal E}$.  Let $\alpha: {\mathcal A} \times \Delta^1_s \to {\mathcal E}$ be a lift in the diagram
\[
\xymatrix{
{\mathcal A} \ar[r]^{u'}\ar[d]^{i_0} & {\mathcal E} \ar[d]^p\\
{\mathcal A} \times \Delta^1_s \ar[r]^{h} & {\mathcal B},
}
\]
where $i_0$ is the inclusion at $0$.  One then defines $[u] \cdot [h] = [w]$, where $w: {\mathcal A} \to {\mathcal F}$ is the unique map such that $i \circ w = \alpha \circ i_1$.  By \cite[Theorem 6.2.1]{Hovey}, this construction gives a well-defined right action of $[{\mathcal A},\Omega {\mathcal B}]_{\aone}$ on $[{\mathcal A},{\mathcal F}]$.  In this situation, there is a boundary morphism $\delta: \Omega^1_s {\mathcal B} \to {\mathcal F}$ defined as the composite
\begin{equation}
\label{eqn:boundaryhomomorphism}
\delta: \Omega^1_s {\mathcal B} {\longrightarrow} {\mathcal F} \times \Omega^1_s {\mathcal B} \longrightarrow {\mathcal F},
\end{equation}
where the first map is the product of the inclusion of the base-point and the identity map, and the second map is the action map just constructed.

\begin{defn}
\label{defn:aonefibration}
An {\em $\aone$-fiber sequence} is a diagram of pointed spaces ${\mathcal X} \to {\mathcal Y} \to \mathcal{Z}$ together with a right action of ${\bf R}\Omega^1_s {\mathcal Z}$ on ${\mathcal X}$ that is isomorphic in $\hop{k}$ to a diagram ${\mathcal F} \to {\mathcal E} \stackrel{p}{\to} {\mathcal B}$ where $p$ is an $\aone$-fibration of (pointed) $\aone$-fibrant spaces with ($\aone$-homotopy) fiber ${\mathcal F}$ together with the action of $\Omega^1_s {\mathcal B}$ on ${\mathcal F}$ discussed above.
\end{defn}

Fiber sequences in topology give rise to long exact sequences in homotopy groups; this result can be generalized to the context of an arbitrary pointed model category.  Applying this observation in the context of $\aone$-homotopy theory the next result is a consequence of (the dual of) \cite[Proposition 6.5.3]{Hovey} together with a sheafification argument.  (Note: a corresponding result holds for fiber sequences in the simplicial homotopy category as well.)

\begin{lem}
\label{lem:fibersequencelongexactsequence}
If ${\mathcal X} \to {\mathcal Y} \to \mathcal{Z}$ is an $\aone$-fiber sequence, then there is a long exact sequence in $\aone$-homotopy sheaves
\[
\cdots \longrightarrow \bpi_{i+1}^{\aone}({\mathcal Z}) \stackrel{\delta_*}{\longrightarrow} \bpi_i^{\aone}({\mathcal X}) \longrightarrow \bpi_i^{\aone}({\mathcal Y}) \longrightarrow \bpi_i^{\aone}({\mathcal Z}) \longrightarrow \cdots,
\]
where $\delta_*$ is the map on $\aone$-homotopy sheaves induced by the morphism $\delta$ of \textup{Equation \ref{eqn:boundaryhomomorphism}}, the sequence terminates with $\bpi_0^{\aone}({\mathcal Z})$ (here, the expression ``long exact sequence" is modified in a standard fashion for $i = 0$ or $1$ where the constituents are groups or pointed sets).
\end{lem}

\subsubsection*{Postnikov towers in $\aone$-homotopy theory}
\begin{defn}
\label{defn:aonepostnikov}
An {\em $\aone$-Postnikov tower} for a pointed space $(\mathcal{X},x)$ consists of a sequence of pointed spaces $(\mathcal{X}^{(n)},x)$, together with pointed maps $i_n: \mathcal{X} \to \mathcal{X}^{(n)}$, and pointed maps $p_n: \mathcal{X}^{(n)} \to \mathcal{X}^{(n-1)}$ having the following properties.
\begin{itemize}
\item[i)] The morphisms $p_n$ are all $\aone$-fibrations.
\item[ii)] The morphisms $i_n: \mathcal{X} \to \mathcal{X}^{(n)}$ induce isomorphisms of sheaves ${i_n}_*: \bpi_m^{\aone}(\mathcal{X}) \to \bpi_m^{\aone}(\mathcal{X}^{(n)})$ for $m \leq n$.
\item[iii)] The sheaves $\bpi_m^{\aone}(\mathcal{X}^{(n)})$ are trivial for $m > n$.
\item[iv)] The induced map $\mathcal{X} \to \operatorname{holim}_n \mathcal{X}^{(n)}$ is an $\aone$-weak equivalence.
\end{itemize}
\end{defn}

\begin{thm}[Morel-Voevodsky]
\label{thm:postnikovtowers}
An $\aone$-Postnikov tower exists, functorially in the input space ${\mathcal X}$.
\end{thm}

\begin{proof}[Sketch of proof.]
The proof of this fact is really a construction.  The {\em simplicial} Postnikov tower is constructed for any space ${\mathcal X}$ in \cite[\S 2 pp. 57-61]{MV}; roughly speaking this provides all of the statements in the simplicial homotopy category.  

Consider the simplicial Postnikov tower of $L_{\aone}{\mathcal X}$; we write $(L_{\aone}{\mathcal X})^{(i)}$ for the $i$-th stage of this tower.  Morel and Voevodsky show that the map
\[
L_{\aone}{\mathcal X} \longrightarrow \operatorname{holim}_i (L_{\aone}{\mathcal X}^{(i)})
\]
is a simplicial weak equivalence (that is the content of \cite[Definition 1.31]{MV}, that the Nisnevich topology satisfies this hypothesis is \cite[Theorem 1.37]{MV}).

By definition, $\bpi_i^{\aone}({\mathcal X}) = \bpi_i^s(L_{\aone}{\mathcal X})$.  By Theorem \ref{thm:pi1stronglyaoneinvariant} and Theorem \ref{thm:stronglyaoneinvariantabelian}, we know that $\bpi_1^{\aone}({\mathcal X})$ is strongly $\aone$-invariant and $\bpi_i^{\aone}({\mathcal X})$ is strictly $\aone$-invariant for any $i \geq 2$.  Equivalently, $B\bpi_1^{\aone}({\mathcal X})$ is $\aone$-local, and $K(\bpi_i^{\aone}({\mathcal X}),j)$ is $\aone$-local for any $i \geq 2$ and every $j \geq 0$.  By induction, one can then replace the map $(L_{\aone}{\mathcal X})^{(i)} \to (L_{\aone}{\mathcal X})^{(i-1)}$ by an $\aone$-fibration without changing the $\aone$-homotopy fiber.
\end{proof}

\subsubsection*{Looping and $\aone$-localization}
If $\mathcal{X}$ is a simplicially fibrant (but not necessarily $\aone$-local space), then there is a canonical pointed map $\mathcal{X} \to L_{\aone}{\mathcal X}$ that induces a morphism of spaces
\[
\Omega^1_s \mathcal{X} \longrightarrow L_{\aone}\Omega^1_s \mathcal{X} \longrightarrow \Omega^1_s L_{\aone}\mathcal{X}.
\]
While the first morphism is an $\aone$-weak equivalence by its very definition, the second morphism is not an $\aone$-weak equivalence in general.  Since $\bpi_0^{\aone}(\Omega^1_s L_{\aone}\mathcal{X}) = \bpi_1^{\aone}(\mathcal{X})$, by Theorem \ref{thm:pi1stronglyaoneinvariant}, a necessary condition that this morphism be an $\aone$-weak equivalence is that $\bpi_0^{\aone}(\Omega^1_s \mathcal{X})$ is a strongly $\aone$-invariant sheaf of groups.  Morel proves that this condition is also sufficient.

\begin{thm}[{\cite[Theorem 6.46]{MField}}]
\label{thm:pathloopsfibration}
The morphism $L_{\aone}\Omega^1_s \mathcal{X} \to \Omega^1_s L_{\aone}\mathcal{X}$ is an $\aone$-weak equivalence if and only if the sheaf of groups $\bpi_0^{\aone}(\Omega^1_s \mathcal{X})$ is strongly $\aone$-invariant.
\end{thm}

\subsubsection*{$\aone$-covering spaces}
One consequence of the existence of the $\aone$-Postnikov tower is the existence of an $\aone$-universal covering space and a corresponding collection of results one refers to as $\aone$-covering space theory.  For more details about the constructions below, see \cite[\S 7.1]{MField}.

\begin{defn}
\label{defn:aoneuniversalcover}
If $\mathcal{X}$ is a space, an {\em $\aone$-cover} of $\mathcal{X}$ is a morphism $\mathcal{X}' \to \mathcal{X}$ that has the unique right lifting property with respect to morphisms that are simultaneously $\aone$-weak equivalences and cofibrations.  If $\mathcal{X}$ is an $\aone$-connected space, then an {$\aone$-universal cover} of $\mathcal{X}$ is a space $\tilde{\mathcal{X}}$ that is $\aone$-connected and $\aone$-simply connected.
\end{defn}

\begin{rem}
By their very definition, $\aone$-covers are $\aone$-fibrations.  Therefore, they give rise to $\aone$-fiber sequences.  We will use this observation repeatedly in the sequel.
\end{rem}

If $(\mathcal{X},x)$ is a pointed $\aone$-connected space, then the $\aone$-Postnikov tower shows that $\mathcal{X}^{(0)}$ is $\aone$-weakly equivalent to the chosen base-point, and $\mathcal{X}^{(1)}$ is $\aone$-weakly equivalent to $B\bpi_1^{\aone}({\mathcal X},x)$.  Morel deduces from these observations (see \cite[Theorem 7.8]{MField}) that the $\aone$-homotopy fiber of the map $\mathcal{X} \to B\bpi_1^{\aone}({\mathcal X},x)$ is an $\aone$-universal cover of $\mathcal{X}$.  More precisely, for any $\aone$-connected space $\mathcal{X}$, an $\aone$-universal covers exists, and the construction is functorial in the input space.

Using Theorem \ref{thm:simplicialhomotopyclassificationoftorsors} and the fact that $B\bpi_1^{\aone}({\mathcal X},x)$ is $\aone$-local, one observes that $\tilde{\mathcal{X}} \to \mathcal{X}$ is a (``Nisnevich locally trivial") $\bpi_1^{\aone}(\mathcal{X},x)$-torsor over $\mathcal{X}$.  Said differently, $\tilde{\mathcal{X}}$ admits a free transitive right action of $\bpi_1^{\aone}({\mathcal X},x)$, which one can think of as the action by ``deck transformations."

\begin{rem}
\label{rem:functorialactionofpi1}
Suppose ${\mathcal X}$ is an $\aone$-connected space.  Fix a base-point $x \in \mathcal{X}(k)$ and set $\bpi := \bpi_1^{\aone}({\mathcal X},x)$.  Since the $\aone$-universal cover $\tilde{\mathcal{X}} \to {\mathcal X}$ is an $\aone$-simply connected space equipped with a free right action of $\bpi$, functoriality of the Postnikov tower shows that the Postnikov tower of $\tilde{{\mathcal X}}$ admits a right $\bpi$-action.  Keeping track of the action of $\bpi$ gives rise to a {\em twisted} Postnikov tower of $\mathcal{X}$, which will be necessary for doing obstruction theory later.  This twisted Postnikov tower is explained in the simplicial setting in \cite[VI.4-5]{GoerssJardine}.
\end{rem}

\begin{lem}
\label{lem:aonecoveringsdontchangehigherhomotopy}
If $f: (\mathcal{Y}',y') \to (\mathcal{Y},y)$ is a pointed $\aone$-cover, then the induced map
\[
\bpi_i^{\aone}(\mathcal{Y'}) \longrightarrow \bpi_i^{\aone}(\mathcal{Y})
\]
is an epimorphism for $i \geq 1$ and an isomorphism for $i > 1$.
\end{lem}

\begin{proof}
We can assume without loss of generality that $\mathcal{Y}'$ and $\mathcal{Y}$ are $\aone$-fibrant.  In that case, the $\aone$-homotopy fiber of the map $f$ is precisely the actual fiber of this morphism.  Let $\mathcal{F} = \operatorname{hofib}(f)$.  Since pullbacks of $\aone$-coverings are $\aone$-coverings, it follows that $\mathcal{F}$ is an $\aone$-covering of $\Spec k$.  The unique right lifting property then implies that any pointed morphism $S^i_s \to \mathcal{F}$ is $\aone$-homotopically constant if $i > 0$, and so the canonical morphism $\mathcal{F} \to \bpi_0^{\aone}(\mathcal{F})$ is an isomorphism.  The result of the claim follows from the long exact sequence in $\aone$-homotopy sheaves of a fibration.
\end{proof}

The following result, which is a straightforward consequence of obstruction theory and the fact that $\bpi_1^{\aone}({\mathcal X})$ is strongly $\aone$-invariant (equivalently $B\bpi_1^{\aone}(\mathcal{X})$ is $\aone$-local), provides a universality property of $\bpi_1^{\aone}(\mathcal{X})$; see \cite[Lemma B.7]{MField}.

\begin{thm}
\label{thm:pi1initialstronglyaoneinvariant}
If $(\mathcal{X},x)$ is a pointed $\aone$-connected space, and $G$ is any strongly $\aone$-invariant sheaf of groups, then the morphism
\[
[(\mathcal{X},x),BG]_{\aone} \longrightarrow \hom_{\Gr^{\aone}_k}(\bpi_1^{\aone}(X),G)
\]
induced by evaluation on $\bpi_1^{\aone}$ is a bijection.
\end{thm}

One consequence of this theorem is the following result that is referred to as the {\em unstable $\aone$-connectivity theorem}: a simplicially $i$-connected space is $\aone$-$i$-connected.

\begin{cor}[{\cite[Theorem 6.38]{MField}}]
\label{cor:unstablehigherconnectivitytheorem}
Suppose $i \geq 0$ is an integer.  If $({\mathcal X},x)$ is a pointed simplicially $i$-connected space (i.e., $\bpi_j^s({\mathcal X},x) = 0$ for all $j \leq i$), then $L_{\aone}{\mathcal X}$ is simplicially $i$-connected.
\end{cor}

\begin{proof}
If ${\mathcal X}$ is a pointed $0$-connected space, then $L_{\aone}{\mathcal X}$ is pointed and $0$-connected by \cite[\S 2 Corollary 3.22]{MV}.  Therefore, suppose ${\mathcal X}$ is a pointed and simplicially $1$-connected space.  Since $L_{\aone}{\mathcal X}$ is simplicially $0$-connected, it suffices to prove $\bpi_1^{\aone}({\mathcal X},x) = \bpi_1^s(L_{\aone}{\mathcal X})$ is trivial.  Equivalently, it suffices by Theorem \ref{thm:pi1initialstronglyaoneinvariant} and the Yoneda lemma to prove that for any strongly $\aone$-invariant sheaf of groups $G$, the set $[(\mathcal{X},x),BG]_{\aone}$ is trivial.  If $G$ is strongly $\aone$-invariant, $BG$ is $\aone$-local, so the canonical map $[({\mathcal X},x),BG]_s \to [({\mathcal X},x),BG]_{\aone}$ is an isomorphism.  However, since ${\mathcal X}$ is simplicially $1$-connected, it follows that the canonical map $[({\mathcal X},x),BG]_s \to \hom(\bpi_1^s({\mathcal X},x),G)$ is a bijection as well (the homomorphisms on the right hand side are taken in the category of Nisnevich sheaves of groups).  Since ${\mathcal X}$ is simplicially $1$-connected, this pointed set is trivial.

There are a number of ways to prove the result for $i \geq 2$, but all the methods we know implicitly involve Theorem \ref{thm:stronglyaoneinvariantabelian}.  We proceed by induction on $i$.  Suppose $L_{\aone}{\mathcal X}$ is simplicially $(i-1)$-connected and ${\mathcal X}$ is simplicially $i$-connected.  It suffices by the Yoneda lemma to show that if $A$ is an arbitrary strictly $\aone$-invariant sheaf of groups, then $\hom_{\Ab^{\aone}_k}(\bpi_i^{\aone}({\mathcal X}),A)$ is trivial.  Indeed, using the fact that $L_{\aone}({\mathcal X})$ is simplicially $(i-1)$-connected, one can show that the induced map
\[
[{\mathcal X},K(A,i)]_{\aone} \to \hom_{\Ab^{\aone}_k}(\bpi_i^{\aone}({\mathcal X}),A)
\]
is a bijection (one argues using obstruction theory; see again \cite[Lemma B.7]{MField} or \cite[Theorem 3.30]{ADExcision}).  In that case, the assumption that $A$ be strictly $\aone$-invariant is equivalent to $K(A,i)$ being $\aone$-local.  Therefore, the map $[{\mathcal X},K(A,i)]_{s} \to [{\mathcal X},K(A,i)]_{\aone}$ is a bijection, and the set on the left hand side is trivial since ${\mathcal X}$ is simplicially $i$-connected.
\end{proof}

\subsubsection*{Properties of the $\aone$-fundamental group}
If $T$ is a split torus, $T$ is strongly $\aone$-invariant.  Indeed, $T$ is $\aone$-invariant since there are no nonconstant maps from $\aone$ to $T$ (in fact, $T$ is $\aone$-rigid in the sense of \cite[\S 3 Example 2.4]{MV}), and $H^1_{\Nis}(\cdot,T)$ is $\aone$-invariant by homotopy invariance of the Picard group.  The following result encodes some facts about the role of torus torsors in $\aone$-covering space theory.

\begin{prop}
\label{prop:fundamentalgroupproperties}
Assume $(X,x)$ is a pointed $\aone$-connected smooth scheme, and let $T$ be a split torus.
\begin{itemize}
\item[i)] There is a canonical isomorphism $H^1_{\Nis}(X,T) \isomt \hom_{\Gr_k}(\bpi_1^{\aone}(X,x),T)$.
\item[ii)] If $f: \tilde{X} \to X$ is a $T$-torsor with $\tilde{X}$ also $\aone$-connected, then for $\tilde{x} \in \tilde{X}(k)$ satisfying $f(\tilde{x}) = x$, there is a short exact sequence
    \[
    1 \longrightarrow \bpi_1^{\aone}(\tilde{X},\tilde{x}) \longrightarrow \bpi_1^{\aone}(X,x) \longrightarrow \bpi_0^{\aone}(T) \longrightarrow 1, and
    \]
\item[iii)] there are isomorphisms $\bpi_i^{\aone}(\tilde{X},\tilde{x}) \isomt \bpi_i^{\aone}(X,x)$ for all $i > 1$.
\end{itemize}
\end{prop}

\begin{proof}
The first statement is a consequence of Theorem \ref{thm:pi1initialstronglyaoneinvariant} using the fact that $T$ is an abelian sheaf of groups.  In that case, the set $H^1_{\Nis}(X,T) = [X,BT]_{\aone}$ is the quotient of $[(X,x),BT]_{\aone}$ by the induced conjugation action of $T(k)$, which is trivial since $T$ is abelian.  The second statement follows immediately from Lemma \ref{lem:fibersequencelongexactsequence} and Theorem \ref{thm:simplicialhomotopyclassificationoftorsors}.  The third statement follows immediately from Lemma \ref{lem:aonecoveringsdontchangehigherhomotopy}.
\end{proof}

\begin{rem}
\label{rem:pi1nontrivial}
An immediate consequence of the first observation is that the $\aone$-fundamental group of a (strictly positive dimensional) smooth proper $\aone$-connected scheme is always non-trivial: take $T = \gm$ and use the fact that the Picard group of a (strictly positive dimensional) smooth proper $k$-scheme is always non-trivial.
\end{rem}

\subsubsection*{$\aone$-homotopy theory of projective spaces and punctured affine spaces}
We now recall some results, due to Morel, regarding the structure of the $\aone$-homotopy groups of projective space.  The results are essentially a breezy overview of \cite[\S 7.3]{MField} introducing the main objects we will need in subsequent sections.

The standard open cover of $\pone$ by two copies of $\aone$ with intersection $\gm$ realizes $\pone$ as the colimit of the diagram $\aone \leftarrow \gm \to \aone$.  The natural map from the homotopy colimit of this diagram to the colimit of this diagram is an $\aone$-weak equivalence, and yields an $\aone$-weak equivlence $\Sigma^1_s \gm \cong \pone$  \cite[\S 3 Corollary 2.18]{MV}.  Using induction and a similar open covering by two sets, one shows that ${\mathbb A}^n \setminus 0$ is $\aone$-weakly equivalent to $\Sigma^{n-1}_s \gm^{\wedge n}$ \cite[\S 3 Example 2.20]{MV}.  It follows that ${\mathbb A}^n \setminus 0$ is simplicially $(n-2)$-connected.  By the unstable $\aone$-connectivity theorem (see Corollary \ref{cor:unstablehigherconnectivitytheorem}), it follows that ${\mathbb A}^n \setminus 0$ is also $(n-2)$-$\aone$-connected.

Morel gives a description of the $(n-1)$st $\aone$-homotopy group of ${\mathbb A}^n \setminus 0$ in terms of so-called Milnor-Witt K-theory sheaves (see \cite[\S 4]{MICM} or \cite[\S 3]{MField} for details regarding this theory).  One could take the following result as a definition of Milnor-Witt K-theory sheaves, and this point of view will suffice for much of the paper.  However, the results of \cite[\S 3]{MField} actually provide a concrete description of the sections of this sheaf over fields, which completely determines the sheaf since it is strictly $\aone$-invariant.

\begin{prop}[{\cite[Theorem 6.40]{MField}}]
\label{prop:aonehomotopygroupsofpuncturedaffinespace}
For every integer $n \geq 2$, there is a canonical isomorphism $\bpi_{n-1}^{\aone}({\mathbb A}^n \setminus 0) \isomt \K^{\MW}_n$.
\end{prop}

Combining Propositions \ref{prop:fundamentalgroupproperties} and \ref{prop:aonehomotopygroupsofpuncturedaffinespace}, one can deduce the following computations of the $\aone$-homotopy groups of projective spaces; these properties will be used repeatedly in Section \ref{s:splittingobstructions}.

\begin{lem}
\label{lem:homotopygroupsofprojectivespace}
Suppose $n$ is an integer $\geq 2$.  There are isomorphisms
\[
\bpi_i^{\aone}({\mathbb P}^n) = \begin{cases}
\gm & \text{ if } i = 1 \\
0 & \text{ if } 1 < i < n, \text{ and }\\
\K^{\MW}_{n+1} & \text{ if } i = n.
\end{cases}
\]
\end{lem}

\begin{rem}
\label{rem:rpncomparison}
For the sake of comparison, we note that the above computation bears a formal resemblance to the computation of the homotopy groups of ${\mathbb{RP}}^{n-1}$.  Indeed, the structure of the group $\pi_i({\mathbb{RP}}^{n-1})$ ($i > 0$) depends on $n$.  If $n = 2$, then ${\mathbb{RP}}^{n-1} = S^1$.  In that case, $\pi_i({\mathbb{RP}}^{1})$ is non-vanishing only for $i = 1$, in which case it is isomorphic to $\Z$.  If $n > 2$, then the canonical map $S^{n-1} \to {\mathbb{RP}}^{n-1}$ is a covering space and gives identifications $\pi_1({\mathbb{RP}}^{n-1}) = \Z/2$ and $\pi_i(S^{n-1}) = \pi_i({\mathbb{RP}}^{n-1})$ for $i > 1$.  In particular, $\pi_i({\mathbb{RP}}^{n-1})$ vanishes in the range $2 < i < n-1$.
\end{rem}

\begin{rem}
\label{rem:homotopygroupsproducts}
The fact that one can choose $L_{\aone}$ to commute with formation of finite products implies that the $i$-th $\aone$-homotopy group of a product of $\aone$-connected spaces is the product of the $i$-th $\aone$-homotopy groups of the constituent factors (one can check the isomorphism on stalks and reduce to the corresponding fact for simplicial sets).
\end{rem}

The homotopy colimit description of $\pone$ above gives the identification $\bpi_1^{\aone}(\pone) = \bpi_1^{\aone}(\Sigma^1_s \gm)$.  Any space of simplicial dimension $0$ (e.g., a sheaf of groups) is simplicially fibrant by \cite[\S 2 Proposition 1.13]{MV}.  Using this fact, Corollary \ref{cor:loopsclassifyingadjointness}, and the loop-suspension adjunction there is a sequence of bijections
\[
\hom_{\Spc_{k,\bullet}}(\gm,G) \longrightarrow [\gm,G]_s \longrightarrow [\gm,\Omega^1_s BG]_s \longrightarrow [\Sigma^1_s \gm,(BG,\ast)]_s
\]
for any Nisnevich sheaf of groups $G$.

If $G$ is, in addition, strongly $\aone$-invariant, there is an identification $[\Sigma^1_s \gm,(BG,\ast)]_s = [\Sigma^1_s \gm,(BG,\ast)]_{\aone}$.  Precomposing the (bijective) morphism of Theorem \ref{thm:pi1initialstronglyaoneinvariant} with the sequence of maps described in the previous paragraph gives the following result.

\begin{lem}
\label{lem:freestronglyaoneinvariant}
If $G$ is a strongly $\aone$-invariant sheaf of groups, then the morphism
\[
\hom_{\Spc_{k,\bullet}}(\gm,G) \longrightarrow \hom_{\Gr^{\aone}_k}(\bpi_1^{\aone}(\Sigma^1_s \gm),G)
\]
described in the previous two paragraphs is a bijection, functorially in $G$.
\end{lem}

This identification can be interpreted as saying $\bpi_1^{\aone}(\pone)$ is the free strongly $\aone$-invariant sheaf of groups generated by $\gm$, and for that reason one uses the following notation.

\begin{notation}
\label{notation:fundamentalgroupofpone}
Set $\Faone(1) := \bpi_1^{\aone}(\pone)$.
\end{notation}

\begin{lem}
There is a short exact sequence of sheaves of groups of the form
\[
1 \longrightarrow \K^{\MW}_2 \longrightarrow \Faone(1) \longrightarrow \gm \longrightarrow 1
\]
The second to last arrow on the right admits a set-theoretic splitting.
\end{lem}

\begin{proof}
The usual $\gm$-torsor ${\mathbb A}^2 \setminus 0 \to \pone$ is an $\aone$-covering space and thus gives rise to an $\aone$-fiber sequence.  The first statement then follows immediately by combining Propositions \ref{prop:fundamentalgroupproperties} and  \ref{prop:aonehomotopygroupsofpuncturedaffinespace}.

Taking $G = \Faone(1)$ in Lemma \ref{lem:freestronglyaoneinvariant}, the identity map $\Faone(1) \to \Faone(1)$ corresponds to a pointed morphism of sheaves of sets $\gm \to \Faone(1)$.  Functoriality of the construction in Lemma \ref{lem:freestronglyaoneinvariant} applied to the morphism of strongly $\aone$-invariant sheaves of groups $\Faone(1) \to \gm$ shows that the composite map $\gm \to \gm$ is the identity map and thus the claimed morphism is in fact a set-theoretic splitting.
\end{proof}

The set-theoretic section gives an isomorphism of sheaves of sets $\psi: \K^{\MW}_2 \times \gm \isomt \Faone(1)$.  The group structure on $\Faone(1)$ is then specified by a factor set, i.e., a morphism of sheaves $\Phi: \gm \times \gm \to \K^{\MW}_2$ (satisfying an appropriate cocycle condition) by means of the formula
\[
\psi(a,x)\psi(b,y) = (ab\Phi(xy),xy).
\]
Morel identifies this factor set explicitly. To this end, he constructs a symbol morphism $\sigma_2: \gm \wedge \gm \to \K^{\MW}_2$ (see \cite[Theorem 3.37]{MField}) that can be precomposed with the epimorphism $\gm \times \gm \to \gm \wedge \gm$ that collapses $\gm \vee \gm$; we refer to this composite morphism also as the symbol morphism.

\begin{thm}[{\cite[Theorem 7.29]{MField}}]
\label{thm:centralextension}
The sheaf $\Faone(1)$ is a central extension of $\gm$ by $\K^{\MW}_2$ with corresponding factor set $\Phi: \gm \times \gm \to \K^{\MW}_2$ given by the symbol morphism.
\end{thm}

\subsubsection*{Automorphisms of some $\aone$-homotopy sheaves}
When we study group structures on $\aone$-homotopy sheaves in Section \ref{s:splittingobstructions}, we will need some information about the sheaf of automorphisms of various homotopy sheaves.  While most cases that appear are short enough to be treated without disturbing the exposition, the automorphisms of $\bpi_1^{\aone}(\pone)$ require more care.

The sheaf of endomorphisms of $\Faone(1)$ admits a rather explicit description.  Indeed, we know that $\hom_{\Gr^{\aone}_k}(\Faone(1),\Faone(1)) = \hom_{\Spc_{k,\bullet}}(\gm,\Faone(1))$ since $\Faone(1)$ is the free strongly $\aone$-invariant sheaf of groups on $\gm$.  The sheaf of morphisms from $\gm$ to $\Faone(1)$ admits a particularly nice description in terms of what are called ``contractions."  If $\F$ is a sheaf of pointed sets, set $\F_{-1} := \underline{\hom}_{\Spc_{k,\bullet}}(\gm,\F)$ (see \cite[Definition 4.3.10]{MIntro}); this construction is functorial by its very definition and by \cite[Lemma 7.33]{MField} respects exact sequences of strongly $\aone$-invariant sheaves of groups.  Using these observations, one can prove the following result.

\begin{lem}[{\cite[Corollary 7.34]{MField}}]
\label{lem:autfaone1}
There is a canonical isomorphism $\Faone(1)_{-1} \cong \Z \oplus \K^{\MW}_1$, and $\Aut(\Faone(1))$ is a sheaf of abelian groups.
\end{lem}

\begin{rem}
A direct description of the ring structure on $\Z \oplus \K^{\MW}_1$ requires some information about the Hopf map \cite[\S 7.3]{MField}.
\end{rem}

\section{Projective bundles and other $\aone$-fiber sequences}
\label{s:fibersequences}
In this section, we fix our notation for the study of projective bundles.  After reviewing some basic facts and terminology regarding projective bundles, we describe a number of results on $\aone$-fibrations associated with torsors.  The section ends with a review of some results of M. Wendt, drawing heavily on work of F. Morel, demonstrating that Zariski locally trivial $PGL_n$-torsors and associated projective bundles give rise to $\aone$-fiber sequences.  The remainder of the material in the section introduces techniques that allow us to effectively study the aforementioned fiber sequences (when the base of the fibration is $\aone$-connected), and these techniques will be used in the computations in Section \ref{s:splittingobstructions}.  Much of this discussion involves relating the fiber sequences for $PGL_n$-torsors, $GL_n$-torsors and $SL_n$-torsors.

\subsubsection*{Zariski locally trivial projective bundles}
Suppose $X$ is a smooth scheme and $\mathcal{E}$ is a finite rank locally free sheaf of $\O_X$-modules on $X$; we will refer to $\mathcal{E}$ as a vector bundle on $X$.  We also systematically abuse notation and identify ${\mathcal E}$ with the corresponding geometric vector bundle $\Spec \operatorname{Sym}^{\bullet} {\mathcal E}^{\vee} \to X$.  As usual, write ${\mathbb P}({\mathcal E})$ for the associated projective bundle with projection morphism ${\mathbb P}({\mathcal E}) \to X$.  Recall the following properties of projective bundles.

\begin{thm}[{\cite[\S 4.1-4.2]{EGAII}}]
Suppose $X$ is a scheme and ${\mathcal E}$ is a rank $n$ vector bundle on $X$.
\begin{itemize}
\item If $f: X' \to X$ and $\mathcal{E}' = f^*{\mathcal E}$, then there is a canonical isomorphism $X' \times_X {\mathbb P}({\mathcal E}) \cong {\mathbb P}({\mathcal E}')$.
\item If $\mathcal{L}$ is a line bundle on $X$, there is a canonical isomorphism ${\mathbb P}({\mathcal E}) \isomt {\mathbb P}(\mathcal{E} \tensor \mathcal{L})$
\item The set of sections of ${\mathbb P}({\mathcal E})$ is canonically in bijection with the set of locally-free subsheaves $\F \subset \mathcal{E}$ such that $\mathcal{E}/\mathcal{F}$ is invertible.
\end{itemize}
\end{thm}

\begin{defn}
\label{defn:homotopysection}
Suppose $\mathcal{E}$ is a rank $n$ vector bundle on a smooth $k$-scheme $X$.  An {\em $\aone$-homotopy section} of ${\mathbb P}({\mathcal E}) \to X$ consists of a smooth scheme $X'$, an $\aone$-weak equivalence $f: X' \to X$ and a subsheaf $\F \subset f^*\mathcal{E}$ with invertible quotient $f^*\mathcal{E}/\F$.
\end{defn}

\begin{rem}
It is not clear (to us) whether existence of an $\aone$-homotopy section is a strictly weaker condition than existence of a section.  Since the Picard group is $\aone$-homotopy invariant, the invertible quotient $f^*\mathcal{E}/\F$ on $X'$ gives rise to an invertible sheaf $\L$ on $X$.  Even if $f$ is flat (which is not an obvious consequence of the assumptions) the situation is unclear.  Indeed, by the assumptions on $X$ and $X'$, the theory of faithfully flat descent implies that
\[
\hom_{\O_X}(\mathcal{E},\L) \longrightarrow \hom_{\O_{X'}}(f^*\mathcal{E},f^*\mathcal{E}/\F)
\]
is injective (identified with the subspace of homomorphism preserving the descent datum), but it is not clear whether the morphism in question descends.  Nevertheless, it seems reasonable to speculate that existence of an $\aone$-homotopy section is equivalent to existence of a section of the pullback to an affine vector bundle torsor over $X$.
\end{rem}

\subsubsection*{Projective bundles over an $\aone$-connected base}
Suppose $X$ is a smooth scheme.  Since $PGL_n$ is a smooth group scheme, \cite[Theorem 11.8]{BrauerIII} guarantees that $PGL_n$-torsors are \'etale locally trivial.  There is a short exact sequence of algebraic groups (i.e., a short exact sequence of \'etale sheaves of groups)
\[
1 \longrightarrow \gm \longrightarrow GL_n \longrightarrow PGL_n \longrightarrow 1.
\]
The associated long exact sequence in (non-abelian) \'etale cohomology yields the sequence of maps
\[
\cdots \longrightarrow H^1_{\et}(X,GL_n) \longrightarrow H^1_{\et}(X,PGL_n) \stackrel{\psi}{\longrightarrow} H^2_{\et}(X,\gm),
\]
where, as usual, the first two terms are pointed sets.

The inclusion map $H^1_{\operatorname{Zar}}(X,GL_n) \to H^1_{\et}(X,GL_n)$ is an isomorphism, i.e., $GL_n$ is a special group in the sense of Serre.  The projective bundles corresponding to $PGL_n$-torsors in the image of the first map are precisely the projectivizations of vector bundles on $X$.  As a consequence, a projective bundle on $X$ is the projectivization of a vector bundle if and only if the image of the corresponding element of $H^1_{\et}(PGL_n)$ under $\psi$ is trivial in $H^2_{\et}(X,\gm)$.  The following result explains why we restrict to Zariski locally trivial projective space bundles in the remainder of the paper.

\begin{prop}
\label{prop:projectivebundlesonaoneconnectedvars}
Suppose $X$ is an $\aone$-connected smooth scheme (so $X(k)$ is non-empty by \textup{\cite[\S 3 Remark 2.5]{MV}}) and $\operatorname{char}(k)$ is not divisible by $n$.  Every projective bundle on $X$ of rank $n-1$ (i.e., with $n-1$-dimensional projective space fibers) that becomes trivial upon restriction to some $k$-rational point $x \in X(k)$ is the projectivization of a vector bundle.
\end{prop}

\begin{proof}
If $k$ is not algebraically closed, there may exist non-trivial \'etale locally trivial projective bundles over $\Spec k$, so the condition regarding triviality upon restriction to a $k$-rational point is clearly necessary.  Write $Br'(X)$ for the prime to $p$ part of the cohomological Brauer group of a smooth scheme $X$.   By \cite[Proposition 4.1.2]{AM}, if $X$ is $\aone$-connected, the canonical map $Br'(k) \to Br'(X)$ is an isomorphism.  Equivalently, for any integer $n$ coprime to $\operatorname{char}(k)$ and any element $[\mathcal{P}]$ of $H^1_{\et}(X,PGL_n)$ represented by a torsor $P \to X$, the class $\psi[\mathcal{P}]$ is pulled back from a class in $H^2_{\et}(\Spec k,\gm)$.

If $x \in X(k)$ is a $k$-rational point such that the restriction of $P$ to $x$ is trivial, then it follows that the original class $\psi[P]$ is trivial, and so $P$ is $PGL_n$-torsor obtained from a $GL_n$-torsor.  In other words, the associated projective space bundle is the projectivization of a vector bundle.  Note that the maps $Br'(k) \to Br'(X)$ induced by different $k$-rational points actually coincide, so the choice of $k$-rational point is irrelevant.
\end{proof}

\subsubsection*{$\aone$-fiber sequences and torsors}
Principal $G$-bundles are among the most common examples of fibrations in topology.  Analogously, Morel gave hypotheses under which Nisnevich locally trivial $G$-torsors give rise to $\aone$-fiber sequence.

\begin{thm}[{\cite[Theorem 6.50]{MField}}]
\label{thm:torsorsgivefibrations}
Suppose ${\mathcal X}$ is a $k$-space.  If $G$ is a Nisnevich sheaf of groups such that $\bpi_0^{\aone}(G)$ is strongly $\aone$-invariant, and if ${\mathcal P} \to {\mathcal X}$ is a (Nisnevich locally trivial) $G$-torsor on ${\mathcal X}$, then
\[
{\mathcal P} \longrightarrow {\mathcal X} \longrightarrow BG
\]
is an $\aone$-fiber sequence.  
\end{thm}

\begin{rem}
\label{rem:actionofloopsofthebase}
In the sequel, when we use this theorem, it will be important to understand the action of ${\mathbf R}\Omega^1_s BG$ on ${\mathcal P}$.  By Corollary \ref{cor:loopsclassifyingadjointness}, we know that ${\mathbf R}\Omega^1_s BG \isomt G$.  The $G$-action on ${\mathcal P}$ coming from its definition as a $G$-torsor is the required action.  To see that this is the case requires tracing through Morel's proof.  Under the hypothesis that $\bpi_0^{\aone}(G)$ is strongly $\aone$-invariant, Theorem \ref{thm:pathloopsfibration} shows that the map $L_{\aone}{\bf R}\Omega^1_s BG \to {\bf R}\Omega^1_s L_{\aone} BG$ is an $\aone$-weak equivalence.  In other words, there is an $\aone$-weak equivalence $L_{\aone}G \isomt {\bf R}\Omega^1_s L_{\aone}BG$.  The action of $G$ on ${\mathcal P}$ stemming from the fact that ${\mathcal P}$ is a $G$-torsor thus yields an action of $L_{\aone}(G)$ on $L_{\aone}{\mathcal P}$.  Applying $L_{\aone}$ to the map ${\mathcal X} \to BG$ makes this map into an $\aone$-fibration, and one identifies $L_{\aone}{\mathcal P}$ with the fiber of this morphism. 
\end{rem}

\subsubsection*{Some models of classifying spaces}
The spaces $BGL_n$ and $BSL_n$ have, in addition to the simplicial models mentioned before, geometric models.  We use the construction from \cite[\S 4.2]{MV}, sometimes called Totaro's model for the classifying space \cite[Remark 1.4]{Totaro}.  If $V$ is the standard $n$-dimensional representation of $GL_n$, then $V$ can be viewed as a representation of $SL_n$ by restriction.  For any integer $N \geq 0$, $GL_n$ (or $SL_n$) acts scheme-theoretically freely on the open subscheme $V_N \subset \hom_k(k^{N+n},V)$ whose points correspond to surjective linear maps.  The quotient $V_N/GL_n$ exists as a smooth scheme and is precisely the Grassmannian $Gr_{n,N}$.  Likewise, the quotient $V_N/SL_n$ exists as a smooth scheme and is the total space of a $\gm$-torsor over $Gr_{n,N}$.  There are transition morphisms $V_N \to V_{N+1}$ that yield morphisms $V_N/GL_n \to V_{N+1}/GL_n$ and $V_N/SL_n \to V_{N+1}/SL_n$ and we set $Gr_n = \colim_N V_N/GL_n$.  Likewise, set $B_{gm}SL_n = \colim_N V_N/SL_n$.  By definition $B_{gm}SL_n \to Gr_n$ is a $\gm$-torsor.

\begin{prop}[{\cite[\S 4 Propositions 2.3 and 2.6]{MV}}]
The space $\colim_N V_N$ is $\aone$-contractible, and the simplicial homotopy class of maps $B_{gm}SL_n \to BSL_n$ (resp. $Gr_n \to BGL_n$) classifying the $SL_n$-torsor $\colim_N V_N \to B_{gm}SL_n$ (resp. $\colim_N V_N \to Gr_n$) is an $\aone$-weak equivalence.
\end{prop}

The map $SL_n \to PGL_n$ is a $\mu_n$-torsor.  If $n$ does not divide $\operatorname{char} k$, then $\mu_n$ is a finite \'etale group scheme, and \'etale locally trivial $\mu_n$-torsors can be described by $\aone$-homotopy classes of maps in $\ho{k}$.  Indeed, under the stated hypotheses on the characteristic, a simplicial classifying space for \'etale locally trivial $\mu_n$-torsors, denoted $B_{\et}\mu_n$, is obtained by pushing forward the classifying space for the \'etale sheaf of groups $\mu_n$ from the \'etale to the Nisnevich topology; see \cite[\S 4]{MV} for details of this construction.  So long as $n$ is not divisible by $\operatorname{char} k$, the space $B_{\et}\mu_n$ is $\aone$-local by \cite[\S 4 Proposition 3.1]{MV}, and for any smooth scheme $U$ an adjunction argument shows (see \cite[\S 4 Proposition 1.16]{MV}) that
\[
[U,B_{\et}\mu_n]_{\aone} = H^1_{\et}(U,\mu_n).
\]
There are various geometric models for $B_{\et}\mu_n$ that are $\aone$-weakly equivalent, but there is one that will be particularly useful for us.

Consider the right action of $SL_n$ on $PGL_n$ where the center $\mu_n \subset SL_n$ acts trivially.  If $ESL_n$ is an $\aone$-contractible space with free $SL_n$-action (e.g., the space $\colim_n V_N$ discussed above), we can consider the space $(PGL_n \times ESL_n)/SL_n$ (here we take the Nisnevich sheaf quotient).  The projection map $PGL_n \to \ast$ gives rise to a morphism
\[
(PGL_n \times ESL_n)/SL_n \longrightarrow BSL_n.
\]
The quotient homomorphism $SL_n \to PGL_n$ gives rise to a morphism
\begin{equation}
\label{eqn:extendingsequenceforpgln}
ESL_n \cong (SL_n \times ESL_n)/SL_n \longrightarrow (PGL_n \times ESL_n)/SL_n
\end{equation}
that is by construction an \'etale locally trivial $\mu_n$-torsor.  Therefore, we can pick a morphism $(PGL_n \times ESL_n)/SL_n \to B_{\et}\mu_n$ classifying this \'etale locally trivial $\mu_n$-torsor.

\begin{lem}
\label{lem:modelforbetmun}
If $\operatorname{char} k$ does not divide $n$, any representative of the simplicial homotopy class of maps
\[
(PGL_n \times ESL_n)/SL_n \longrightarrow B_{\et}\mu_n
\]
classifying the $\mu_n$-torsor of \textup{Morphism \ref{eqn:extendingsequenceforpgln}} is an $\aone$-weak equivalence.
\end{lem}

\begin{proof}
The statement above is actually independent of the model of $ESL_n$ chosen.  The projection morphism $\pi: PGL_n \times ESL_n \to (PGL_n \times ESL_n )/SL_n$ is an epimorphism of Nisnevich sheaves, and thus the map $\breve{C}(\pi) \to (PGL_n \times ESL_n)/SL_n$ is a simplicial weak equivalence by \cite[\S 2 Lemma 1.15]{MV}.  Now, by definition $\breve{C}(\pi)_i \cong PGL_n \times ESL_n \times SL_n^{\times i+1}$.

Let $E'SL_n$ be another $\aone$-contractible space with free $SL_n$-action, and write $\pi': PGL_n \times E'SL_n \to (PGL_n \times E'SL_n)/SL_n$ for the associated quotient morphism.  Also, let $\pi'': PGL_n \times ESL_n \times E'SL_n \to (PGL_n \times ESL_n \times E'SL_n)/SL_n$ be the quotient morphism.  Now the projections $p: PGL_n \times ESL_n \times E'SL_n \to PGL_n \times ESL_n$ and $p':PGL_n \times ESL_n \times E'SL_n \to PGL_n \times E'SL_n$ induce morphisms of simplicial sheaves
\[
\breve{C}(\pi) \longleftarrow \breve{C}(\pi'') \longrightarrow \breve{C}(\pi').
\]
Since both $ESL_n$ and $E'SL_n$ are $\aone$-contractible by assumption, all of these morphisms are termwise $\aone$-weak equivalences, and therefore $\aone$-weak equivalences by \cite[\S 2 Proposition 2.14]{MV}.

Since we can pick the model of $ESL_n$, let us use Totaro's model from the proof of Lemma \ref{lem:slnglnrelation}.  In that case the result is precisely \cite[\S 4 Lemma 2.5]{MV}.
\end{proof}

\subsubsection*{Some $\aone$-homotopy theory of $BSL_n$, $BGL_n$}
Corollary \ref{cor:pi0linearalgebraicgroup} implies that the classifying spaces $BSL_n$, $BGL_n$ and $BPGL_n$ are all $\aone$-connected.  In fact, the same proof shows that classifying spaces have a canonical base-point.  Likewise, we will always point sheaves of groups by their identity section.

\begin{lem}
\label{lem:bslnaonesimplyconnected}
The space $SL_n$ is $\aone$-connected, and $BSL_n$ is $\aone$-$1$-connected.
\end{lem}

\begin{proof}
That $\bpi_0^{\aone}(SL_n) = 1$ is a consequence of the fact that for any field $L$, $SL_n(L)$ is generated by elementary matrices and that $\aone$-connectedness can be checked on sections over finitely generated extensions of the base field (see, e.g., \cite[Proposition 2.2.7]{AM} for the second statement).  Since $\bpi_0^{\aone}(SL_n) = 1$, it follows that $SL_n$-torsors give rise to $\aone$-fiber sequences by Theorem \ref{thm:torsorsgivefibrations}.  Since $ESL_n$ is $\aone$-contractible, the second result follows from the first by considering the $\aone$-fiber sequence $SL_n \to ESL_n \to BSL_n$.
\end{proof}

Functoriality of the classifying space construction applied to the determinant homomorphism $GL_n \to \gm$ and the inclusion $SL_n \hookrightarrow GL_n$ give rise to a sequence of maps
\[
BSL_n \to BGL_n \to B\gm;
\]
a key property of this sequence of maps is summarized in the following result.

\begin{lem}
\label{lem:slnglnrelation}
There is an $\aone$-fiber sequence of the form $BSL_n \to BGL_n \to B\gm$.
\end{lem}

\begin{proof}
The map $B_{gm}SL_n \to Gr_n$ is a $\gm$-torsor and is therefore classified by a map $Gr_n \to B\gm$.  Furthermore, $\gm$-torsors are $\aone$-covering spaces and in particular $\aone$-fibrations, so it follows that the $\aone$-homotopy fiber of $B_{gm}SL_n \to Gr_n$ is precisely $\gm$.  Since $B\gm$ is already $\aone$-local, the action ${\bf R}\Omega^1_s L_{\aone} B\gm$ on $BSL_n$ is precisely the action of $\gm$ on $L_{\aone} BSL_n$.
\end{proof}

\begin{prop}
\label{prop:slnglnrelation}
The determinant induces an isomorphism $\bpi_1^{\aone}(BGL_n) \cong \gm$, and for every $i > 1$, the inclusion map $SL_n \hookrightarrow GL_n$ induces isomorphisms $\bpi_i^{\aone}(BSL_n) \isomt \bpi_i^{\aone}(BGL_n)$.
\end{prop}

\begin{proof}
Consider the long exact sequence in $\aone$-homotopy groups stemming from the $\aone$-fiber sequence of Lemma \ref{lem:slnglnrelation}.  By \cite[\S 4 Proposition 3.8]{MV} we know that $\bpi_1^{\aone}(B\gm) = \gm$ and $\bpi_i^{\aone}(B\gm) = 0$, for $i > 1$ and therefore $\bpi_i^{\aone}(BSL_n) \isomt \bpi_i^{\aone}(BGL_n)$ in this range.  Likewise, since $\bpi_1^{\aone}(BSL_n) = 1$ by Lemma \ref{lem:bslnaonesimplyconnected}, the map $\bpi_1^{\aone}(BGL_n) \to \bpi_1^{\aone}(B\gm) = \gm$ is an isomorphism.
\end{proof}

\begin{thm}[{\cite[Theorem 7.20]{MField}}]
\label{thm:homotopygroupsofsln}
Suppose $n$ is an integer $\geq 2$.  There are isomorphisms
\[
\bpi_1^{\aone}(SL_n) \cong \begin{cases}
\K^{\MW}_2 & \text{ if } n = 2, \text{ and }\\
\K^M_2 & \text{ if } n \geq 3.
\end{cases}
\]
Furthermore, for every integer $i \geq 0$ the maps $\bpi_{i+1}^{\aone}(BSL_n) \to \bpi_i^{\aone}(SL_n)$ induced by the fiber sequence $SL_n \longrightarrow ESL_n \longrightarrow BSL_n$ are isomorphisms.
\end{thm}

\begin{rem}
\label{rem:slnstabilization}
As we observed above, Theorem \ref{thm:torsorsgivefibrations} implies that $SL_n$-torsors give rise to $\aone$-fiber sequences.  In particular, there is an $\aone$-fiber sequence of the form
\[
SL_n \longrightarrow SL_{n+1} \longrightarrow SL_{n+1}/SL_n.
\]
One can identify $SL_{n+1}/SL_n$ with the smooth affine quadric in ${\mathbb A}^{2n+2}$ defined by the equation $\sum_{i=1}^{n+1} x_iy_i = 1$.  It is then straightforward to check that projection onto $(x_1,\ldots,x_n)$ is a Zariski locally trivial smooth morphism with affine space fibers, and therefore determines an $\aone$-weak equivalence $SL_{n+1}/SL_n \to {\mathbb A}^{n+1} \setminus 0$.  The computations of Proposition \ref{prop:aonehomotopygroupsofpuncturedaffinespace} show that if $n \geq 3$, the map $\bpi_1^{\aone}(SL_n) \to \bpi_1^{\aone}(SL_{n+1})$ is an isomorphism.  On the other hand the $\aone$-weak equivalence $SL_2 \to {\mathbb A}^2 \setminus 0$ gives a description of $\bpi_1^{\aone}(SL_2)$ by Proposition \ref{prop:aonehomotopygroupsofpuncturedaffinespace}.  Theorem \ref{thm:homotopygroupsofsln} is really a statement about the morphism $\bpi_1^{\aone}(SL_2) \to \bpi_1^{\aone}(SL_3)$; we will come back to this point later.
\end{rem}

\begin{rem}
\label{rem:aonesimplyconnectedcoverofbgln}
The map $BGL_n \to B\gm$ identifies $B\gm$ with the first stage of the $\aone$-Postnikov tower for $BGL_n$.  Thus, $BSL_n$ is the $\aone$-simply connected cover of $BGL_n$.
\end{rem}

\subsubsection*{Strong $\aone$-invariance of $\bpi_0^{\aone}(PGL_n)$}
We now prove that $\bpi_0^{\aone}(PGL_n)$ is strongly $\aone$-invariant, which by Theorem \ref{thm:torsorsgivefibrations} implies that $PGL_n$-torsors yield $\aone$-fiber sequences.  To see this, we will simply compute $\bpi_0^{\aone}(PGL_n)$.  We observed in Lemma \ref{lem:modelforbetmun} that the space $(PGL_n \times ESL_n)/SL_n$ is a model for $B_{\et}\mu_n$ and that projection onto the second factor provided a morphism $(PGL_n \times ESL_n)/SL_n \to BSL_n$.

\begin{lem}
\label{lem:slnpglnfibersequence}
If $\operatorname{char} k$ does not divide $n$, there is an $\aone$-fiber sequence of the form
\[
SL_n \longrightarrow PGL_n \longrightarrow B_{\et}\mu_n.
\]
\end{lem}

\begin{proof}
Under the hypothesis, the map $E_{\et}\mu_n \to B_{\et}\mu_n$ is even an $\aone$-covering space (see \cite[Lemma 7.5.2]{MField} for this result).  Therefore, there exists an $\aone$-fiber sequence
\[
\mu_n \longrightarrow E_{\et}\mu_n \longrightarrow B_{\et}\mu_n.
\]
Pick a fibrant model of $B_{\et}\mu_n^f$ of $B_{\et}\mu_n$ and choose an explicit map $PGL_n \to B_{\et}\mu_n$ representing the $\mu_n$-torsor $SL_n \to PGL_n$.  By construction $SL_n$ is the pullback of $E_{\et}\mu_n^f \to B_{\et}\mu_n^f$ along the map $PGL_n \to B_{\et}\mu_n$, and thus this map is an $\aone$-fibration.  The fiber of this map is $\mu_n$ and being already simplicially fibrant and $\aone$-local is the $\aone$-homotopy fiber of the map $SL_n \to PGL_n$.

Since $B_{\et}\mu_n$ is $\aone$-local, it follows that $\Omega^1_s B_{\et}\mu_n^f$ is also $\aone$-local, and we can identify it with the space $\mu_n$ (this is proven along the same lines as Corollary \ref{cor:loopsclassifyingadjointness} using the discussion at the top of \cite[p. 131]{MV}).  The action of $\Omega^1_s B_{\et}\mu_n^f$ on $SL_n$ is precisely the (left or right) action of $\mu_n \subset SL_n$.
\end{proof}

\begin{lem}
\label{lem:betmuntoslnfibrationsequence}
If $\operatorname{char} k$ does not divide $n$, there is an $\aone$-fiber sequence of the form
\[
PGL_n \longrightarrow B_{\et}\mu_n \longrightarrow BSL_n.
\]
\end{lem}

\begin{proof}
Since $\bpi_0^{\aone}(SL_n) = 1$, and the trivial group is strongly $\aone$-invariant, it follows from Theorem \ref{thm:torsorsgivefibrations} that $SL_n$-torsors give rise to $\aone$-fiber sequences.  Consider the $\aone$-fiber sequence given by the universal $SL_n$-torsor $SL_n \to ESL_n \to BSL_n$.

In Lemma \ref{lem:modelforbetmun}, we constructed an $\aone$-weak equivalence $(PGL_n \times ESL_n)/SL_n \to B_{\et}\mu_n$.  Consider the projection map $(PGL_n \times ESL_n)/SL_n \to ESL_n/SL_n \cong BSL_n$, and pullback the universal $SL_n$-torsor over $BSL_n$ by means of an explicit representative of this morphism.  The result is still an $SL_n$-torsor and therefore an $\aone$-fiber sequence of the form
\[
PGL_n \times ESL_n \longrightarrow (PGL_n \times ESL_n)/SL_n \longrightarrow BSL_n
\]
The middle term is already $\aone$-weakly equivalent to $B_{\et}\mu_n$ and projection map $PGL_n \times ESL_n \to PGL_n$ is an $\aone$-weak equivalence since $ESL_n$ is $\aone$-contractible.  The boundary map $SL_n \cong {\bf R}\Omega^1_s L_{\aone} BSL_n \to PGL_n$ can be identified with the morphism $SL_n \to PGL_n$.
\end{proof}

Let ${\mathcal H}^1_{\et}(\mu_n)$ be the Nisnevich sheaf associated with the presheaf $U \mapsto H^1_{\et}(U,\mu_n)$; this sheaf is precisely $\bpi_0^{\aone}(B_{\et}\mu_n)$ by the discussion prior to Lemma \ref{lem:slnpglnfibersequence}.  As an immediate consequence of the long exact sequence in the homotopy groups of the fibration in Lemma \ref{lem:betmuntoslnfibrationsequence}, and the fact that $BSL_n$ is $\aone$-connected, we deduce the following result.

\begin{cor}
\label{cor:strongaoneinvariancepi0ofpgln}
If $\operatorname{char} k$ does not divide $n$, the map $\bpi_0^{\aone}(PGL_n) \to \bpi_0^{\aone}(B_{\et}\mu_n) = {\mathcal H}^1_{\et}(\mu_n)$ induced by the classifying map for $\mu_n$-torsor $SL_n \to PGL_n$ is an isomorphism.  In particular, $\bpi_0^{\aone}(PGL_n)$ is a strongly $\aone$-invariant sheaf of groups.
\end{cor}

\begin{proof}
The only statement that does not follow immediately from Lemma \ref{lem:betmuntoslnfibrationsequence} is the fact that ${\mathcal H}^1_{\et}(\mu_n)$ is strongly $\aone$-invariant.  In fact, this sheaf is a strictly $\aone$-invariant Nisnevich sheaf with transfers; see \cite[Lemmas 9.23-24]{MVW}.
\end{proof}

\begin{rem}
If $G$ is a split, semi-simple algebraic group, then Morel has proven that $\bpi_0^{\aone}(G)$ is strongly $\aone$-invariant \cite[Theorem A.2]{MFM}.  Our proof is a slightly more geometric and explicit version of his proof.
\end{rem}

\subsubsection*{$\aone$-fiber sequences of associated bundles}
Suppose $G$ is a Nisnevich sheaf of groups and ${\mathcal F}$ is a space with a left $G$-action.  Given a space ${\mathcal X}$ and a (right) $G$-torsor ${\mathcal P} \to {\mathcal X}$, we can form the contracted product
\[
{\mathcal P} \times^G {\mathcal F} := {\mathcal P} \times {\mathcal F}/G,
\]
where the quotient is taken in the category of spaces; a space of the form ${\mathcal P} \times^G {\mathcal F}$ is called an associated fiber bundle.

If $G \to {\mathcal P} \to {\mathcal X}$ is an $\aone$-fiber sequence, then we can consider the associated sequence of morphisms
\[
{\mathcal F} \longrightarrow {\mathcal P} \times^G {\mathcal F} \longrightarrow {\mathcal X}.
\]
An action of $G$ on $\F$ gives a homomorphism $G \to \Aut(\F)$.  Now, by assumption, there is an action of $\Omega^1_s {\mathcal X}$ on $G$, and by composition we obtain a map $\Omega^1_s {\mathcal X} \times G \to G \to \Aut(\F)$.  Define an action $\Omega^1_s {\mathcal X}$ on ${\mathcal F}$ as the composite
\[
\Omega^1_s {\mathcal X} \times {\mathcal F} \longrightarrow \Aut(\F) \times {\mathcal F} \longrightarrow {\mathcal F}.
\]
If we know that ${\mathcal F}$ is the $\aone$-homotopy fiber of ${\mathcal P} \times^G {\mathcal F} \to {\mathcal X}$, then it follows that this sequence is also an $\aone$-fiber sequence.  This can actually be checked in a number of situations of interest, as was demonstrated by Morel and Wendt.

Suppose $(X,x)$ is a connected and pointed smooth scheme, and $\mathcal{E} \to X$ is a rank $n+1$ vector bundle on $X$ with zero section $i: X \to \mathcal{E}$.  We write $\mathcal{E}^{\circ}$ for the space $\mathcal{E} \setminus i(X)$.  The induced morphism $\mathcal{E}^{\circ} \to X$ is a Zariski locally trivial smooth morphism with fibers isomorphic to punctured affine spaces.  Fix base-points in ${\mathbb A}^{n+1} \setminus 0$ and ${\mathbb P}^n$ making the usual map ${\mathbb A}^{n+1} \setminus 0 \to {\mathbb P}^n$ into a pointed map.  The usual fiber of ${\mathcal E}^{\circ}$ (resp. ${\mathbb P}({\mathcal E})$) at $x$ is ${\mathbb A}^{n+1}\setminus 0$ (resp. ${\mathbb P}^n$) and we use the image of the chosen base-point in ${\mathcal E}^{\circ}$ (resp. ${\mathbb P}({\mathcal E})$) to obtain a base-point of the latter space.  From these constructions, we obtain $\aone$-fiber sequences.

\begin{prop}[{\cite[Propositions 5.2 and 6.2]{Wendt}}]
\label{prop:projectivebundlefibration}
Suppose $(X,x)$ is a pointed smooth scheme and $\mathcal{E}$ is a rank $(n+1)$-vector bundle over $X$.  The sequences of pointed spaces
\[
\begin{split}
&{\mathbb A}^{n+1} \setminus 0 \longrightarrow {\mathcal E}^{\circ} \longrightarrow X, \text{ and } \\
&{\mathbb P}^n \longrightarrow {\mathbb P}({\mathcal E}) \longrightarrow X,
\end{split}
\]
(induced by the projections and inclusions discussed above) are $\aone$-fiber sequences.
\end{prop}

\begin{rem}
The proof of this result relies on the auxiliary fact that $BSing_*^{\aone}(G)$ is $\aone$-local for $G = SL_n, GL_n, or PGL_n$; this fact is established in \cite[Theorem 1.5]{MFM}.  See \cite[Theorem 5.3]{Wendt} for a more detailed version of the proof presented by Morel.
\end{rem}

The existence of this fiber sequence, together with the fact that ${\mathbb A}^{n+1} \setminus 0$ and ${\mathbb P}^n$ are $\aone$-connected (in the former case, one needs to assume $n \geq 1$) has the following consequence, which we use without mention in the sequel.

\begin{cor}
If $\mathcal{E}$ is a vector bundle of rank $n \geq 2$ on an $\aone$-connected smooth scheme $X$, then both ${\mathcal E}^{\circ}$ and ${\mathbb P}({\mathcal E})$ are $\aone$-connected.
\end{cor}

The vector bundle ${\mathcal E}$ over $X$ of the statement determines a simplicial homotopy class of maps $X \to BGL_{n+1}$.  If we assume $X$ is $\aone$-connected, we can assume that this map is pointed.  In that case, Corollary \ref{cor:loopsclassifyingadjointness} shows that the map $X \to BGL_{n+1}$ is equivalent to a simplicial homotopy class of maps
\[
{\bf R}\Omega^1_s L_{\aone}X \longrightarrow GL_{n+1},
\]
and this morphism coincides with the connecting homomorphism for the $\aone$-fiber sequence of the $GL_{n+1}$-torsor defined by ${\mathcal E}$.  Here, we have used the fact that $GL_{n+1}$, being a space of simplicial dimension $0$, is automatically simplicially fibrant.


Similarly, the (Zariski and hence Nisnevich locally trivial) $PGL_{n+1}$-torsor underlying the associated projective bundle is then induced by a simplicial homotopy class of maps $X \to BPGL_{n+1}^f$; by adjunction this map corresponds to a simplicial homotopy class of maps ${\bf R}\Omega^1_s X \to PGL_{n+1}$.  This discussion allows us to give a description of the connecting homomorphism $\delta$ from Equation \ref{eqn:boundaryhomomorphism}: up to $\aone$-homotopy, the connecting map in the $\aone$-fiber sequence of Proposition \ref{prop:projectivebundlefibration} is induced by the composite map
\[
{\bf R}\Omega^1_s X \longrightarrow GL_{n+1} \longrightarrow PGL_{n+1} \longrightarrow PGL_{n+1}/P \cong {\mathbb P}^n.
\]

\begin{cor}
\label{cor:projectivebundlelongexactsequence}
Suppose $(X,x)$ is a pointed smooth scheme and $\mathcal{E}$ is a rank $n+1$ vector bundle over $X$, there is a long exact sequence in $\aone$-homotopy groups of the form
\[
\cdots \longrightarrow \bpi_{i+1}^{\aone}(X,x) \stackrel{\delta_*}{\longrightarrow} \bpi_i^{\aone}({\mathbb P}^n) \longrightarrow \bpi_i^{\aone}({\mathbb P}({\mathcal E})) \to \bpi_i^{\aone}(X,x) \longrightarrow \cdots.
\]
Furthermore, if ${\mathbb P}({\mathcal E})$ admits an $\aone$-homotopy section (see \textup{Definition \ref{defn:homotopysection}}, then the morphism $\bpi_i^{\aone}({\mathbb P}(\mathcal{E})) \to \bpi_i^{\aone}(X)$ is split.
\end{cor}

\begin{proof}
For the long exact sequence, simply combine Lemma \ref{lem:fibersequencelongexactsequence} and Proposition \ref{prop:projectivebundlefibration}.

If $f: X' \to X$ is an $\aone$-weak equivalence of smooth schemes, it follows from properness of the $\aone$-local model structure that the induced morphism ${\mathbb P}(f^*{\mathcal E}) \to {\mathbb P}({\mathcal E})$ is again an $\aone$-weak equivalence.  Thus, up to $\aone$-weak equivalence, we can replace the fiber sequence in the statement by the corresponding fiber sequence on any $X'$ that is $\aone$-weakly equivalent to $X$.  In particular, if ${\mathbb P}({\mathcal E})$ admits an $\aone$-homotopy section, this gives a morphism $X \to {\mathbb P}({\mathcal E})$ in the $\aone$-homotopy category, which provides the required splitting at the level of homotopy sheaves.
\end{proof}

\section{Splitting behavior of bundles and $\aone$-homotopy groups}
\label{s:splittingobstructions}
In this section, we study the $\aone$-homotopy groups of projectivizations of vector bundles by analyzing the connecting homomorphism and splitting behavior of the attached long exact sequence in $\aone$-homotopy groups of Corollary \ref{cor:projectivebundlelongexactsequence}.  We begin by recalling a topological analog of the problem under consideration, which is the study of homotopy groups of the projectivization of a real rank $n$ ($n \geq 2$) vector bundle on a smooth (not necessarily orientable) connected manifold.  In outline, our analysis of the connecting homomorphism proceeds in analogy with the topological situation, though there are a number of interesting differences stemming from dissimilarities between the classical homotopy theory of ${\mathbb R}{\mathbb P}^n$ and the $\aone$-homotopy theory of ${\mathbb P}^n$.  In a sense, these dissimilarities stem from the fact that the $\aone$-algebraic topology encodes information about the complex points as well.  From the point of view of the $\aone$-fundamental group, the most interesting case is that of rank $2$ bundles: the fundamental group of ${\mathbb R}{\mathbb P}^1 \cong S^1$ is free abelian, but the $\aone$-fundamental group of $\pone$, while still a free sheaf of groups (see Notation \ref{notation:fundamentalgroupofpone}), is not abelian.  Furthermore, while the higher homotopy groups of $S^1$ are trivial, the higher $\aone$-homotopy groups of $\pone$ are not, and this leads to potential higher obstructions to splitting (we discuss this last point at the very end of the section).

We show that in the case of split vector bundles, the higher $\aone$-homotopy groups decompose as direct sums (Proposition \ref{prop:splitvectorbundles}).  We then investigate the case of not necessarily split bundles:  the analysis is different depending on whether the rank of the bundle is equal to $2$ or $\geq 3$.  The second case is easier:  the $\aone$-fundamental group of the projectivization of a vector bundle of rank $\geq 3$ on an $\aone$-connected smooth scheme decomposes as a product in a number of cases; see Theorem \ref{thm:rankn} for more details.  For projectivizations of bundles of rank $2$, we introduce a (twisted) Euler class in Definition \ref{defn:eulerclass}.  In a number of situations, our Euler class turns out to control the connecting homomorphism in the long exact sequence of Corollary \ref{cor:projectivebundlelongexactsequence} and existence of an $\aone$-homotopy section up to the first stage of the $\aone$-Postnikov tower; these results are contained in Theorem \ref{thm:projectivebundletrivialeulerclass}.  Finally, Lemma \ref{lem:higherobstructions} contains some weaker statements about higher $\aone$-homotopy groups of projective bundles.

\subsubsection*{Interlude: topological motivation}
Suppose $E \to M$ is a rank $n$ real vector bundle on a closed connected manifold $M$.  Forming the projective space of lines in $E$ gives an ${\mathbb{RP}}^{n-1}$-bundle over $M$ denoted ${\mathbb P}(E)$.   There is an associated long exact sequence in homotopy groups of a fibration
\[
\cdots \longrightarrow \pi_{i+1}(M) \longrightarrow \pi_i({\mathbb{RP}}^{n-1}) \longrightarrow \pi_i({\mathbb P}(E)) \longrightarrow \pi_i(M) \longrightarrow \pi_{i-1}({\mathbb{RP}}^{n-1}) \longrightarrow \cdots.
\]
The structure of the group $\pi_i({\mathbb{RP}}^{n-1})$ ($i > 0$) depends on $n$.  If $n = 2$, then ${\mathbb{RP}}^{n-1} = S^1$.  In that case, $\pi_i(S^1)$ is non-vanishing only for $i = 1$, and the map $\pi_2(M) \to \pi_1(S^1) = \Z$ is induced by the map $M \to BO(2)$ by evaluation on $\pi_2$.  If $n > 2$, then the canonical map $S^{n-1} \to {\mathbb{RP}}^{n-1}$ is a covering space and gives identifications $\pi_1({\mathbb{RP}}^{n-1}) = \Z/2$ and $\pi_i(S^{n-1}) = \pi_i({\mathbb{RP}}^{n-1})$ for $i > 1$.  In particular, $\pi_i({\mathbb{RP}}^{n-1})$ vanishes in the range $2 < i < n-1$.

We focus on the case where $n = 2$.  If $\tilde{M} \to M$ is the universal cover of $M$, the composite map $\tilde{M} \to M \to BO(2)$ induces the same map upon applying the functor $\pi_2$ as the original map $M \to BO(2)$.  Since $\tilde{M}$ is $1$-connected, one knows that $\hom(\pi_2(M),\Z) = H^2(\tilde{M},\Z)$ by the Hurewicz theorem and the universal coefficient theorem.  The element in the latter group determined by the map $\tilde{M} \to BO(2)$ is the Euler class of the bundle $\tilde{M} \to BO(2)$.  Thus, the connecting homomorphism is trivial if and only if the pullback of $E$ to $\tilde{M}$ has trivial Euler class.  Equivalently, this Euler class can be viewed as a ``twisted" Euler class on $M$.  Indeed, $E$ determines an orientation character $\omega_E: \pi_1(M) \to \pi_1(BO(2)) = \Z/2$ and a corresponding local system $\Z[\omega_E]$, and the Euler class above can be viewed as an element of $H^2(M,\Z[\omega_E])$.

If the pullback of $E$ to $\tilde{M}$ has trivial Euler class, we deduce that there are isomorphisms $\pi_i({\mathbb P}(E)) \isomt \pi_i(M)$ for $i > 1$, and $\pi_1({\mathbb P}(E))$ is an extension of $\pi_1(M)$ by $\Z$.  Non-triviality of the Euler class of $E$ is an obstruction to existence of a nowhere vanishing section.  Note: many expositions of this fact restrict to the oriented situation (see, e.g., \cite[Property 9.7 and Theorem 12.5]{MilnorStasheff}), but everything works if we use the orientation local system described in the previous paragraph.  Triviality of the Euler class together with an obstruction theory argument provides a splitting of the map $\pi_1({\mathbb P}(E)) \to \pi_1(M)$.  The group structure on $\pi_1({\mathbb P}(E))$ is therefore specified by a homomorphism $\pi_1(M) \to Aut(\Z) = \Z/2$, induced by conjugation via the splitting.

The conjugation action just specified has a geometric origin.  The classifying map of the vector bundle $E$ gives a map $\Omega^1 M \to O(2)$ and each element $g$ of $O(2)$ gives a map $S^1 \to S^1$ that does not necessarily preserve the base-point, but each such map gives a morphism $\pi_1(S^1) \to \pi_1(S^1)$, and the automorphism of $\pi_1(S^1)$ so-defined only depends on the class in $\pi_0(O(2))$ of the element $g$, which is precisely the value of $\det g$.  In other words, the homomorphism $\pi_1(M) \to \Z/2$ can be identified with $\omega_E$.

\subsubsection*{Determinants and orientation characters}
We now show how to associate with any algebraic vector bundle an analog of the topological orientation character. Suppose $(X,x)$ is a pointed $\aone$-connected smooth scheme and $\mathcal{E}$ is a rank $(n+1)$ vector bundle over $X$ classified by a pointed simplicial homotopy class of maps $X \to BGL_{n+1}$.  This morphism gives rise to a simplicial homotopy class of maps $X \to L_{\aone}BGL_{n+1}$.  By functoriality of the Postnikov tower, such a morphism induces a map $X^{(1)} \to (BGL_{n+1})^{(1)}$.  By means of the fiber sequence of Lemma \ref{lem:slnglnrelation}, we know that $(BGL_{n+1})^{(1)} \cong B\gm$, and the map $BGL_{n+1} \to B\gm$ is induced by the determinant.  Now, the composite map $X \to B\gm$ is by Theorem \ref{thm:pi1initialstronglyaoneinvariant} and Proposition \ref{prop:fundamentalgroupproperties} equivalent to the homomorphism $\bpi_1^{\aone}(X) \to \gm$ obtained by evaluation on $\bpi_1^{\aone}$ (we use this observation repeatedly in what follows).  In this way, given a vector bundle $\mathcal{E}$, we obtain a homomorphism
\[
\omega_{\mathcal{E}}: \bpi_1^{\aone}(X) \to \gm,
\]
which we call the $\aone$-orientation character associated with $\mathcal{E}$.  An explicit cocycle computation using an open cover of $X$ on which the bundle $\mathcal{E}$ trivializes gives the following result.

\begin{lem}
\label{lem:determinantconstruction}
Under the isomorphism $Pic(X) \isomt \hom(\bpi_1^{\aone}(X),\gm)$ of \textup{Proposition \ref{prop:fundamentalgroupproperties}}, the $\aone$-orientation character $\omega_{\mathcal{E}}$ corresponds to $\det \mathcal{E}$.
\end{lem}

Since the $PGL_{n+1}$-torsor associated with the vector bundle ${\mathcal E}$ is Zariski locally trivial, we have a classifying map $X \to BPGL_{n+1}$ associated with the projective bundle ${\mathbb P}({\mathcal E})$.  By construction, this morphism factors as $X \to BGL_{n+1} \to BPGL_{n+1}$.  Corollary \ref{cor:strongaoneinvariancepi0ofpgln} identifies $\bpi_0^{\aone}(PGL_{n+1}) = {\mathcal H}^1_{\et}(\mu_{n+1})$.  Since $\bpi_0^{\aone}(PGL_{n+1})$ is strongly $\aone$-invariant, we can also identify $\bpi_1^{\aone}(BPGL_{n+1}) = {\mathcal H}^1_{\et}(\mu_{n+1})$ by means of Theorem \ref{thm:torsorsgivefibrations} and the long exact sequence in $\aone$-homotopy sheaves of a fibration.  Therefore, given a vector bundle $\mathcal{E}$, the projectivization determines a homomorphism
\[
\omega_{{\mathbb P}(\mathcal{E})}: \bpi_1^{\aone}(X) \to {\mathcal H}^1_{\et}(\mu_{n+1}).
\]
The next result identifies this homomorphism.

\begin{prop}
\label{prop:modnplus1orientationcharacter}
If $(X,x)$ is a pointed $\aone$-connected smooth scheme, and ${\mathcal E}$ is a rank $(n+1)$-vector bundle on $X$, the map $\omega_{{\mathbb P}(\mathcal{E})}$ associated with ${\mathbb P}({\mathcal E})$ (as described above) is equivalent to the element of $\Pic(X)/(n+1)\Pic(X)$ obtained by reducing $\det({\mathcal E})$ modulo $n+1$.
\end{prop}

\begin{proof}
By means of Theorem \ref{thm:pi1initialstronglyaoneinvariant}, a homomomorphism as in the statement gives a class in $H^1_{\Nis}(X,{\mathcal H}^1_{\et}(\mu_{n+1}))$; this group is canonically identified with $\Pic(X)/(n+1)\Pic(X)$.  By assumption, the map $\bpi_1^{\aone}(X,x) \to \bpi_1^{\aone}(BPGL_{n+1})$ factors through $\bpi_1^{\aone}(BGL_n)$.  Above, we identified the homomorphism $\bpi_1^{\aone}(X,x) \to \gm$ as $\det {\mathcal E}$.  The homomorphism $\gm \to {\mathcal H}^1_{\et}(\mu_{n+1})$ of $\aone$-fundamental sheaves of groups induces a map $\Pic(X) \to \Pic(X)/(n+1)\Pic(X)$ that we claim is precisely reduction modulo $n+1$.

One way to see this is to observe that $GL_{n+1}$ can be identified as a $\gm$-torsor over $PGL_{n+1}$.  Thus, there is an $\aone$-fiber sequence of the form
\[
\gm \longrightarrow GL_{n+1} \longrightarrow PGL_{n+1}.
\]
The associated long exact sequence in $\aone$-homotopy sheaves (at $\bpi_0^{\aone}$) then yields an exact  sequence of (Nisnevich sheaves of groups) of the form
\[
\gm \longrightarrow \gm \longrightarrow {\mathcal H}^1_{\et}(\mu_{n+1}).
\]
The first map is the composite of the inclusion of the center $\gm \to GL_{n+1}$ followed by the projection map $\bpi_0^{\aone}(GL_n) \to \gm$, which sends a matrix to its determinant, and the composite map sends $t \in \gm$ to $t^{n+1}$.  Therefore, the homomorphism $\gm \longrightarrow {\mathcal H}^1_{\et}(\mu_{n+1})$ can be identified with the connecting homomorphism in the Kummer exact sequence.  The induced map on cohomology is precisely the map induced by reducing modulo $n+1$.
\end{proof}

\subsubsection*{Split vector bundles}
We begin by rephrasing the condition that a vector bundle on a scheme splits as a direct sum of line bundles.  To this end, if $T \subset GL_{n+1}$ is a maximal torus of $GL_{n+1}$ there is an induced morphism $BT \to BGL_{n+1}$ (we implicitly assume that the models chosen are fibrant and $\aone$-local below).  Given a space ${\mathcal X}$ we will refer to simplicial homotopy classes of maps $\mathcal{X} \to BGL_{n+1}$ as rank $(n+1)$-vector bundles on ${\mathcal X}$.  We will say that a vector bundle of rank $(n+1)$ on ${\mathcal X}$ {\em splits as a sum of line bundles} if given the map $\mathcal{X} \to BGL_{n+1}$, then we can find a maximal torus $T \subset GL_{n+1}$ and a $\mathcal{X} \to BT$ making the obvious triangle commute.  When $X$ is a smooth scheme, these notions correspond precisely to existence of a splitting as a direct sum of line bundles in the usual sense.

\begin{defn}
If $\mathcal{E}$ is a rank $(n+1)$-vector bundle on $X$, we will say that $\mathcal{E}$ is {\em $\aone$-homotopy split} if we can find a space ${\mathcal X}$ and a morphism of spaces $f: \mathcal{X} \to X$ that is an $\aone$-weak equivalence, such that the composite map $\mathcal{X} \to X \to BGL_{n+1}$ (we write $f^*{\mathcal E}$ for this composite) splits as a sum of line bundles.
\end{defn}

Let $T$ be a maximal torus in $GL_{n+1}$, and let $\gm \subset GL_{n+1}$ be the inclusion of the center.  The composite map $T \to GL_{n+1} \to PGL_{n+1}$ identifies $T/\gm$ with a maximal torus of $PGL_{n+1}$.

\begin{prop}
\label{prop:splitvectorbundles}
If $X$ is an $\aone$-connected smooth scheme over a field $k$, and $\mathcal{E}$ is an $\aone$-homotopy split rank $n+1$ vector bundle on $X$, then the connecting homomorphism $\bpi_{i+1}^{\aone}(X) \to \bpi_i^{\aone}({\mathbb P}^n)$ in the long exact sequence of \textup{Corollary \ref{cor:projectivebundlelongexactsequence}} is the trivial map for $i \geq 1$.  Thus, for every integer $i \geq 1$ there are split short exact sequences of the form
\[
1 \longrightarrow \bpi_i^{\aone}({\mathbb P}^n) \longrightarrow \bpi_i^{\aone}({\mathbb P}({\mathcal E})) \longrightarrow \bpi_i^{\aone}(X) \to 1.
\]
If furthermore $k$ is perfect, for $i \geq 2$, $\bpi_i^{\aone}({\mathbb P}({\mathcal E}))$ is a direct sum.
\end{prop}

\begin{proof}
Since $\mathcal{E}$ is $\aone$-homotopy split, we can find a space $\mathcal{X}$ and a morphism $f: \mathcal{X} \to X$ such that $f^*{\mathcal{E}}$ splits as a sum of line bundles.  In more detail, there exist a maximal torus $T \subset GL_{n+1}$ and a lift of the classifying map $\mathcal{X} \to BGL_{n+1}$ through a morphism $\mathcal{X} \to BT$.

The morphism ${\bf R}\Omega^1_s \mathcal{X} \to {\mathbb P}^n$ inducing the connecting homomorphism comes from the classifying map $\mathcal{X} \to BGL_{n+1}$ by applying the simplicial loops functor.  If the morphism $\mathcal{X} \to BGL_{n+1}$ lifts to $\mathcal{X} \to BT$, then the map displayed above factors through a morphism ${\bf R}\Omega^1_s \mathcal{X} \to T \to PGL_{n+1}$.  Therefore, applying $\bpi_i^{\aone}$, we see that the connecting homomorphism factors as
\[
\bpi_{i}^{\aone}({\bf R}\Omega^1_s \mathcal{X}) \longrightarrow \bpi_i^{\aone}(T) \longrightarrow \bpi_i^{\aone}(PGL_{n+1}).
\]
Since $\bpi_i^{\aone}(T)$ is trivial for $i > 0$, the connecting homomorphisms are trivial in this range, and the long exact sequence in homotopy sheaves splits into short exact sequences.

To obtain the splitting, we proceed as follows.  An inverse of the $\aone$-weak equivalence $\mathcal{X} \to X$ composed with the map $\mathcal{X} \to BT$ determines a collection of line bundles on $X$.  The projective bundle on $\mathcal{X}$ corresponding to this direct sum of line bundles is pulled back from the projectivization of the sum of line bundles on $X$.  By properness of the $\aone$-local model structure, one can check that the two projective bundles in the previous sentence are $\aone$-weakly equivalent.  Now, the projectivization of a direct sum of line bundles on $X$ admits a section, which can be used to split the long exact sequences.

For the final statement, observe that for any pointed space $({\mathcal Y},y)$, the sheaves $\bpi_i^{\aone}({\mathcal Y},y)$ are strictly $\aone$-invariant if $i \geq 2$ (this is the only place where the assumption that $k$ is perfect is used).  The category of strictly $\aone$-invariant sheaves is abelian by \cite[Lemma 6.2.13]{MStable}, and the result follows from the splitting lemma.
\end{proof}

\subsubsection*{$\aone$-fundamental groups of ${\mathbb P}^n$-bundles, $n \geq 2$}
We now formulate hypotheses under which the connecting homomorphism $\bpi_2^{\aone}(X) \to \bpi_1^{\aone}({\mathbb P}^n)$ from Corollary \ref{cor:projectivebundlelongexactsequence} is trivial.  The analysis breaks into two parts based on the structure of the $\aone$-fundamental group of ${\mathbb P}^n$: we deal with $n \geq 2$ first, and later with $n = 1$, which is more interesting.

\begin{lem}
\label{lem:ranknconnectinghomomorphism}
Suppose $n \geq 2$ is an integer, $X$ is an $\aone$-connected smooth scheme, and $\mathcal{E}$ is a rank $n+1$ vector bundle on $X$.  The connecting homomorphism $\bpi_2^{\aone}(X) \to \bpi_1^{\aone}({\mathbb P}^n)$ in \textup{Corollary \ref{cor:projectivebundlelongexactsequence}} is trivial.
\end{lem}

\begin{proof}
The connecting homomorphism in the long exact sequence is induced by the composite
\[
{\mathbf R}\Omega^1_s X \longrightarrow GL_{n+1} \longrightarrow {\mathbb A}^{n+1} \setminus 0 \longrightarrow {\mathbb P}^n,
\]
since the map $GL_{n+1} \to {\mathbb P}^n$ can be factored ($GL_{n+1}$-equivariantly) through ${\mathbb A}^{n+1} \setminus 0$.  Applying $\bpi_1^{\aone}$ to this sequence, we see that the connecting homomorphism factors through $\bpi_1^{\aone}({\mathbb A}^{n+1} \setminus 0)$, which is trivial if $n > 1$.
\end{proof}

\begin{thm}
\label{thm:rankn}
Let $n \geq 2$ be an integer, and let $\mathcal{E}$ be a rank $n+1$ vector bundle on an $\aone$-connected smooth scheme $X$.  If either i) the $\aone$-fundamental group of $X$ is a split $k$-torus, or ii) ${\mathbb P}(\mathcal{E})$ admits an $\aone$-homotopy section, then there is an identification
\[
\bpi_1^{\aone}({\mathbb P}({\mathcal E})) \cong \gm \times \bpi_1^{\aone}(X).
\]
\end{thm}

\begin{proof}
By Lemma \ref{lem:ranknconnectinghomomorphism}, the connecting homomorphism $\bpi_2^{\aone}(X) \to \bpi_1^{\aone}({\mathbb P}^n)$ is trivial.  Assume $\bpi_1^{\aone}(X) = T$, where $T$ is a split $k$-torus.  The extension
\[
1 \longrightarrow \gm \longrightarrow \bpi_1^{\aone}({\mathbb P}({\mathcal E})) \longrightarrow T \longrightarrow 1
\]
is a (Nisnevich locally trivial) $\gm$-torsor over $T$.  Since $T$ is split, and we know that $H^1_{\Nis}(T,\gm) = Pic(T) = 0$, it follows that the extension is trivial.

If ${\mathbb P}({\mathcal E})$ admits an $\aone$-homotopy section, then the short exact sequence in question is split.  Split extensions of $\bpi_1^{\aone}(X)$ by $\gm$ are classified by morphisms $\bpi_1^{\aone}(X) \to \Aut(\gm)$.  We claim that all such extensions are trivial.  We have an isomorphism of sheaves $\Aut(\gm) = \Z/2$, and $\Z/2$ is strongly $\aone$-invariant (if the characteristic of the base field is unequal to $2$) by \cite[\S Proposition 3.5]{MV}.   Since $\Z/2$ is strongly $\aone$-invariant, evaluating a pointed homotopy class of maps $X \to B\Z/2$ on $\bpi_1^{\aone}$ determines a morphism
\[
[(X,x),(B\Z/2,\ast)]_{\aone} \isomto \hom_{\Gr^{\aone}_k}(\bpi_1^{\aone}(X),\Z/2).
\]
and since $\Z/2$ is abelian, we get an identification $[(X,x),(B\Z/2,\ast)]_{\aone} \isomt [X,B\Z/2]_{\aone}$.  However, since $X$ is $\aone$-connected, any Nisnevich locally trivial $\Z/2$-torsor is in fact trivial by \cite[Proposition 4.1.1]{AM}.
\end{proof}

\begin{ex}
\label{ex:ranknexamples}
Condition (i) of the theorem is satisfied for $X = {\mathbb P}^n$, $n \geq 2$ and also for many smooth toric varieties by \cite[Theorem 6.4]{ADExcision}.  Every vector bundle on $\pone$ is a direct sum of line bundles (Grothendieck's theorem \cite[Theorem 2.1.1]{OSS}), and thus the projectivization of any vector bundle on $\pone$ admits a section (and not just an $\aone$-homotopy section).  In other words, condition (ii) is satisfied for any vector bundle on $X = \pone$.
\end{ex}

\subsubsection*{Euler classes: the construction}
We observed in Remark \ref{rem:aonesimplyconnectedcoverofbgln} that the $\aone$-simply connected cover of $BGL_{n+1}$ is precisely $BSL_{n+1}$.  By means of functoriality of the $\aone$-Postnikov tower, the morphism $X \to BGL_{n+1}$ gives rise to a morphism of $\aone$-fiber sequences:
\[
\xymatrix{
\tilde{X} \ar[r]\ar[d] & X \ar[d]\ar[r] & B\bpi_1^{\aone}(X) \ar[d] \\
BSL_{n+1} \ar[r] & BGL_{n+1} \ar[r] & B\gm,
}
\]
Since $BSL_{n+1}$ is $\aone$-simply connected by Proposition \ref{prop:slnglnrelation}, the second stage of the $\aone$-Postnikov tower for the map $\tilde{X} \to BSL_{n+1}$ gives a map
\[
\tilde{X} \to K(\bpi_2^{\aone}(BSL_{n+1}),2),
\]
and this morphism intertwines the $\bpi_1^{\aone}(X)$-action on $\tilde{X}$ (see Remark \ref{rem:functorialactionofpi1}) with the $\gm$-action on $\bpi_2^{\aone}(BGL_{n+1}) = \bpi_2^{\aone}(BSL_{n+1})$ (via Proposition \ref{prop:slnglnrelation}).  As we now explain, this action allows us to produce a certain equivariant cohomology class, which we will see can be viewed as a class on $X$ itself.

For simplicity, we set $\bpi:= \bpi_1^{\aone}(X)$.  We let $B\bpi$ be the simplicial classifying space of $\bpi$ and $E\bpi$ an $\aone$-contractible space with free $\bpi$-action (if $\bpi \to \ast$ is the structure map, we can take $E\bpi$ to be the \u Cech simplicial object associated with this epimorphism).  Now, there is a diagonal $\bpi$-action on $X \times E\bpi$, and we set
\[
\tilde{X}_{\bpi} := (\tilde{X} \times E\bpi)/\bpi.
\]
Analogously, using the $\gm$-action on $\bpi_2^{\aone}(BSL_{n+1})$ we described above, we can form the twisted Eilenberg-MacLane space 
\[
K_{\gm}(\bpi_2^{\aone}(BSL_{n+1}),2) := (K(\bpi_2^{\aone}(BSL_{n+1}),2) \times E\gm)/\gm
\]  
Since the construction of the previous paragraph intertwines the $\bpi$-action on $\tilde{X}$ and the $\gm$-action on $\bpi_2^{\aone}(BSL_{n+1})$ the morphism $\tilde{X} \to K(\bpi_2^{\aone}(BSL_{n+1}),2)$ yields a morphism
\[
\tilde{X}_{\bpi} \longrightarrow K_{\gm}(\bpi_2^{\aone}(BSL_{n+1}),2).
\]
There is a canonical morphism $\tilde{X}_{\bpi} \to X$ defined by ``projection onto the first factor."  If $C$ is the \u Cech simplicial object associated with the projection $\tilde{X} \times E\bpi \to \tilde{X}_{\bpi}$, then it follows that the induced map $C \longrightarrow \tilde{X}_{\bpi}$ is a simplicial weak equivalence.  On the other hand, $C$ is termwise weak equivalent to $X$ and therefore weakly equivalent to $X$.  Thus, the class in question can be viewed as an $\aone$-homotopy class of maps
\[
X \longrightarrow K_{\gm}(\bpi_2^{\aone}(BSL_{n+1}),2).
\]
This class admits a description as a sheaf cohomology class on $X$.

We can take the sheaf of sections of the projection morphism $\tilde{X} \times_{\bpi_1^{\aone}(X)} \bpi_2^{\aone}(BSL_{n+1})$ to obtain a sheaf $\bpi_2^{\aone}(BSL_{n+1})(\det(\mathcal{E}))$ on $X$.  Just like in \cite[\S B.3 (see p. 251)]{MField}, we can view the new map as an element of the group $H^2_{\Nis}(X,\bpi_2^{\aone}(BSL_{n+1})(\det(\mathcal{E})))$.  In particular, note that the class in $H^2_{\Nis}(X,\bpi_2^{\aone}(BSL_{n+1})(\det(\mathcal{E})))$ vanishes if and only if the class in $H^2_{\Nis}(\tilde{X},\bpi_2^{\aone}(BSL_{n+1}))$ vanishes.  Using the computations of Proposition \ref{prop:slnglnrelation} for an explicit description of $\bpi_2^{\aone}(BSL_{n+1})$, we can make the following definition.

\begin{defn}
\label{defn:eulerclass}
If $X$ is an $\aone$-connected smooth scheme, and $\mathcal{E}$ is a rank 2 (resp. $n+1$) vector bundle on $X$, the {\em Euler class}, denoted $e(\mathcal{E})$, is the element in $H^2_{\Nis}(X,\K^{\MW}_2(\det(\mathcal{E})))$ (resp. $H^2_{\Nis}(X,\K^M_2(\det(\mathcal{E})))$) described above.
\end{defn}

\begin{rem}
The construction we have given, with evident changes in notation, defines an Euler class for any vector bundle over an $\aone$-connected space (not necessarily just a scheme).  By functoriality of the universal covering space construction, this construction is functorial with respect to pullbacks of vector bundles along morphisms of $\aone$-connected spaces.
\end{rem}

\begin{ex}
\label{ex:universaleulerclass}
Take the model of $BGL_2$ coming from Totaro's construction; this is an $\aone$-connected space.  Applying the construction above to the universal rank $2$ vector bundle ${\mathcal V}$ on $BGL_2$ gives rise to a universal Euler class in $H^2_{Nis}(BGL_2,\K^{\MW}_2(\det {\mathcal V}))$.  Suppose ${\mathcal E}$ is a rank $2$ vector bundle on a smooth scheme $X$ classified by a morphism $X \to BGL_2^f$.  The homomorphism $\bpi_1^{\aone}(X) \to \gm$ induced by applying $\bpi_1^{\aone}(\cdot)$ to the classifying homomorphism of the vector bundle is precisely the determinant of $\mathcal{E}$ by Lemma \ref{lem:determinantconstruction}.  There is an induced pullback homomorphism
\[
H^2_{\Nis}(BGL_2,\K^{\MW}_2(\det {\mathcal V})) \longrightarrow H^2_{\Nis}(X,\K^{\MW}_2(\det {\mathcal E})).
\]
By construction, the Euler class of Definition \ref{defn:eulerclass} is the image under this pullback homomorphism of the universal Euler class.
\end{ex}

\begin{rem}
If $k$ is an infinite perfect field of characteristic unequal to $2$, the group $H^2_{\Nis}(X,\K^{\MW}_2(\det(\mathcal{E})))$ can be identified with the $\det(\mathcal{E})$-twisted Chow-Witt group $\widetilde{CH}^2(X,\det(\mathcal{E}))$ (see \cite[D\'efinition 1.2]{BargeMorel} and \cite[D\'efinition 10.4.6]{Fasel1}), but we will not prove or use this here.  The stated definition of the Euler class is equivalent to other definitions that appear in the literature (e.g., \cite[Theorem 8.14]{MField} or \cite[D\'efinition 13.2.1]{Fasel1}).  This can be checked in the universal case, but we do not give a detailed verification here since it is not necessary in the sequel.
\end{rem}

\begin{rem}
\label{rem:twistdoesnotmatter}
The situation in rank $\geq 3$ is different from rank $2$.  Indeed, the action of $\gm$ on $\K^M_2$ is given by a homomorphism $\gm \to \Aut(\K^M_2)$.  We know that $\mathbf{End}(\K^M_2,\K^M_2) = \K^M_0 = \Z$, and the subsheaf of automorphisms corresponds to $\pm 1 = \mu_2$.  Therefore, if the base field has characteristic unequal to $2$, the action is given by a homomorphism $\gm \to \mu_2$.  On the other hand, this homomorphism can be interpreted as a Nisnevich locally trivial $\mu_2$-torsor on $BGL_2$ and is therefore trivial by \cite[Proposition 4.1.1]{AM}.  In other words, the twist is trivial {\em independent} of $\mathcal{E}$.  This is a manifestation of the fact that the sheaf $\K^M_2$ is orientable, in the sense that the Hopf map $\eta$ acts trivially (we will not make this precise here, but see \cite[Definition 6.2.5]{MIntro} or \cite[Definition 1.2.7]{DegliseOriented} for a discussion in the setting of stable $\aone$-homotopy).
\end{rem}

We state the following lemma for completeness.

\begin{lem}
\label{lem:splitcasetrivialeulerclass}
If $X$ is an $\aone$-connected smooth scheme, and $\mathcal{E}$ is a split vector bundle on $X$, then $e({\mathcal E})$ is trivial.
\end{lem}

\begin{proof}
Under the hypothesis, the map $X \to BGL_{n+1}$ lifts through a map $X \to BT$.  If $T' = T \cap SL_{n+1}$, then the map $\tilde{X} \to BSL_{n+1}$ lifts through a morphism $\tilde{X} \to BT'$.  By Proposition \ref{prop:fundamentalgroupproperties}, there is a bijection
\[
[\tilde{X},BT']_{\aone} \isomto \hom_{\Gr^{\aone}_k}(\bpi_1^{\aone}(\tilde{X}),T').
\]
However, since $\tilde{X}$ is $\aone$-simply connected the resulting map must be trivial.  Therefore, the defining map of the Euler class factors through an $\aone$-homotopically trivial map, and the Euler class must itself be trivial.
\end{proof}

\subsubsection*{Euler classes and the connecting homomorphism}
\begin{lem}
\label{lem:aonefundamentalgroupeulerclass}
Suppose $X$ is an $\aone$-connected smooth scheme.  Let $\mathcal{E}$ be a rank $2$ vector bundle on $X$ and let ${\mathbb P}(\mathcal{E})$ be the associated projective space bundle.  The connecting homomorphism $\bpi_2^{\aone}(X) \to \bpi_1^{\aone}({\mathbb P}^1)$ is trivial if and only if $e(\mathcal{E})$ is trivial.
\end{lem}

\begin{proof}
As we discussed at the beginning of the section, the connecting homomorphism $\delta: {\bf R}\Omega^1_s L_{\aone}X \to {\mathbb P}^n$ is induced by a map ${\bf R}\Omega^1_s L_{\aone}X \to GL_{n+1}$ that comes from the classifying map of the bundle by looping.  We proceed in a fashion analogous to the proof of Lemma \ref{lem:ranknconnectinghomomorphism}.  If $\tilde{X}$ is the $\aone$-universal cover of $X$, then the morphism $\bpi_2^{\aone}(\tilde{X}) \to \bpi_2^{\aone}(X)$ is an isomorphism.  Moreover, since the projective bundle in question is the projectivization of a vector bundle, the connecting homomorphism lifts through a morphism $\bpi_2^{\aone}(X) \to \bpi_1^{\aone}({\mathbb A}^2 \setminus 0) = \K^{\MW}_2$.  Combining these two observations, we see that the connecting homomorphism under consideration is trivial if and only if the induced morphism $\bpi_2^{\aone}(\tilde{X}) \to \K^{\MW}_2$ is trivial.

Now, since $\tilde{X}$ is $\aone$-simply connected, the canonical map
\[
H^2_{\Nis}(\tilde{X},\K^{\MW}_2) = [\tilde{X},K(\K^{\MW}_2,2)]_{\aone} \longrightarrow \hom_{\Ab^{\aone}_k}(\bpi_2^{\aone}(\tilde{X}),\K^{\MW}_2)
\]
induced by evaluation on $\bpi_2^{\aone}(\cdot)$ is a bijection by \cite[Theorem 3.30]{ADExcision}.  Tracing through our construction of the Euler class, we see that the class in $H^2_{\Nis}(\tilde{X},\K^{MW}_2)$ so obtained, which is necessarily equivariant for the action of $\bpi_1^{\aone}(X)$ on $\tilde{X}$ and an induced action on $\K^{\MW}_2$, descends to the Euler class.
\end{proof}

\begin{rem}
In Remark \ref{rem:twistdoesnotmatter}, we observed that the twist is irrelevant if $\operatorname{rk} \mathcal{E} \geq 3$.  Therefore, assume the projective bundle we considered above is the projectivization of a fixed rank $2$ vector bundle $\mathcal{E}$.  If $\L$ is a line bundle, then the projectivization of $\mathcal{E} \tensor \L$ is isomorphic to that of $\mathcal{E}$.  While the connecting homomorphisms in the long exact sequences in $\aone$-homotopy groups associated with each of these vector bundles are the same, the lifts we choose in the course of the proof of the lemma will differ by the morphism ${\bf R}\Omega^1_s X \to \gm$ obtained by applying simplicial loops to the map $X \to B\gm$ classifying the line bundle $\L$.  As a consequence, to make a statement about the connecting homomorphism that is independent of the choice of a vector bundle ``lift," one must assert that the Euler classes of all $\mathcal{E} \tensor \L$ are all trivial.  Since the line bundle $\det({\mathcal E} \tensor \L)$ is isomorphic to $\det({\mathcal E}) \tensor \L^{\tensor 2}$, these Euler classes live in what appear to be different groups.  However, it is known that when the twists differ by the square of a line bundle, the twisted groups in which the Euler classes lie are canonically isomorphic (this is a manifestation of the quadratic nature of orientation, which is well-known in the theory of Chow-Witt groups \cite{Fasel1}; see also \cite[Definition 4.3]{MField} and the subsequent discussion for related discussion).
\end{rem}

\subsubsection*{$\aone$-fundamental groups of ${\mathbb P}^1$-bundles: trivial Euler class}
In this section, we describe the $\aone$-fundamental group of a $\pone$-bundle with trivial Euler class in a number of situations.  We refer the reader to Morel's computations of $\bpi_1^{\aone}(\pone)$ from \cite[\S 7.3]{MField} for context.

\begin{thm}
\label{thm:projectivebundletrivialeulerclass}
Let $\mathcal{E}$ be a rank $2$ vector bundle on an $\aone$-connected smooth scheme $X$.  Assume $e(\mathcal{E})$ is trivial and either i) $\bpi_1^{\aone}(X)$ is a split torus, or ii) ${\mathbb P}(\mathcal{E})$ admits an $\aone$-homotopy section (again see \textup{Definition \ref{defn:homotopysection}}).  The short exact sequence
\[
1 \longrightarrow \Faone(1) \longrightarrow \bpi_1^{\aone}({\mathbb P}({\mathcal E})) \longrightarrow \bpi_1^{\aone}(X) \longrightarrow 1.
\]
is split, and the group structure on the term in the middle is completely determined by the class of $\det({\mathcal E})$ in $\Pic(X)/2\Pic(X)$.
\end{thm}

\begin{proof}
First, let us show that if $e(\mathcal{E})$ is trivial, then under assumption (i) the extension is split (the extension is clearly split in case (ii)).  We analyze the extension in two stages corresponding to the fact that $\Faone(1)$ is an extension of $\gm$ by $\K^{\MW}_2$.  By assumption $\bpi_1^{\aone}(X)$ is a split torus, call it $T$.  Observe that $\K^{\MW}_2$ is a normal subgroup sheaf of $\bpi_1^{\aone}({\mathbb P}({\mathcal E}))$ and taking the quotient we get a short exact sequence of the form
\[
1 \longrightarrow \gm \longrightarrow \bpi_1^{\aone}({\mathbb P}({\mathcal E}))/\K^{\MW}_2 \longrightarrow T \longrightarrow 1.
\]
Since the Picard group of a split torus is trivial, it follows that this extension is split.  The analysis of the group structure is now the same as in Theorem \ref{thm:rankn}: the group structure is determined by a homomorphism $T \to \Aut(\gm) = \Z/2$, which is trivial since $X$ is $\aone$-connected.  As a consequence, we obtain an isomorphism $\bpi_1^{\aone}({\mathbb P}({\mathcal E}))/\K^{\MW}_2 \isomt T \times \gm$.

Next, consider the short exact sequence
\[
1 \longrightarrow \K^{\MW}_2 \longrightarrow \bpi_1^{\aone}({\mathbb P}({\mathcal E})) \longrightarrow T \times \gm \longrightarrow 1.
\]
The action of $T \times \gm$ on $\K^{\MW}_2$ is induced by a morphism $\gm \to \Aut(\K^{\MW}_2)$ and a morphism $T \to \Aut(\K^{\MW}_2)$.  The morphism $\gm \to \Aut(\K^{\MW}_2)$ induces the group structure on $\Faone(1)$, and thus this sequence is not split.  However, since we already understand the action of $\gm$ on $\K^{\MW}_2$, we will pass to an $\aone$-connected $\aone$-covering space so as to only consider the action of $T$.  To do this, let $i: X \to \mathcal{E}$ is the zero section of $\mathcal{E}$, and consider the $\gm$-torsor over ${\mathbb P}(\mathcal{E})$ corresponding to $\mathcal{E}^{\circ} = \mathcal{E} \setminus i(X)$.  Studying the $T$-action on $\K^{MW}_2$ corresponds to studying long exact sequence in $\aone$-homotopy groups associated with the $\aone$-fiber sequence ${\mathbb A}^2 \setminus 0 \to \mathcal{E}^{\circ} \to X$ of Proposition \ref{prop:projectivebundlefibration}.  In other words, we reduce to considering the $T$-action on the ${\mathbb A}^2 \setminus 0$-bundle from which the $\pone$-bundle is derived.

Now, we want to use the assertion that the Euler class is trivial.  Since the Euler class is constructed equivariantly, let $f: \tilde{X} \to X$ be the $\aone$-universal covering map: by assumption $\tilde{X}$ is a $T$-torsor over $X$ and thus a smooth scheme.  The composite morphism $\tilde{X} \to X \to BGL_2$ factors through the $\aone$-simply-connected cover of $BGL_2$ (again, see Remark \ref{rem:aonesimplyconnectedcoverofbgln}), which we identified with $BSL_2$.  Thus, one obtains a map $\tilde{X} \to BSL_2$ that is $T$-equivariant for a morphism $T \to \gm = \bpi_1^{\aone}(BGL_2)$; this homomorphism corresponds to the line bundle $\det({\mathcal E})$ by Lemma \ref{lem:determinantconstruction}.

Let ${\mathcal E}^{\circ}|_{\tilde{X}}$ denote the pullback of ${\mathcal E}^{\circ}$ under the $\aone$-universal covering map.  The map ${\mathcal E}^{\circ}|_{\tilde{X}} \to \tilde{X}$ is again an $\aone$-fiber sequence by Proposition \ref{prop:projectivebundlefibration} and we can try to lift along this map.  To this end, identify $\tilde{X}^{(1)} = \ast = {\mathcal E}^{\circ}|_{\tilde{X}}^{(0)}$, and consider the lifting problem
\[
\xymatrix{
& {\mathcal E}^{\circ}|_{\tilde{X}}^{(1)}\ar[d] \\
\tilde{X} \ar[r]\ar@{-->}[ur] & \ast
}
\]
coming from the first stage of the $\aone$-Postnikov tower of ${\mathcal E}^{\circ}|_{\tilde{X}}$: we would like to know that the dashed arrow exists.  Looking at the long exact sequence in $\aone$-homotopy sheaves of the $\aone$-fibration ${\mathcal E}^{\circ}|_{\tilde{X}} \to \tilde{X}$, triviality of the Euler class allows us to conclude that the connecting homomorphism $\bpi_2^{\aone}(X) \to \bpi_1^{\aone}({\mathbb A}^2 \setminus 0)$ is trivial.  As a consequence, $\bpi_1^{\aone}({\mathcal E}^{\circ}|_{\tilde{X}}) = \bpi_1^{\aone}({\mathbb A}^2 \setminus 0)$ or, equivalently, there is an identification ${\mathcal E}^{\circ}|_{\tilde{X}}^{(1)} = B(\K^{\MW}_2)$.

Using the Euler class, we can even fix a splitting.  Indeed, the Euler class comes from a map $\tilde{X} \to K(\bpi_2^{\aone}(BSL_2))$ that intertwines the $T$-action on $\tilde{X}$ with the $\gm$-action on $\bpi_2^{\aone}(BSL_2)$.  The induced homomorphism thus give a choice of splitting of the sequence
\[
1 \longrightarrow \K^{\MW}_2 \longrightarrow \bpi_1^{\aone}({\mathcal E}^{\circ}) \longrightarrow T \longrightarrow 1.
\]
Even more, the construction gives a splitting that factors as the composite $T \longrightarrow \gm \to \Aut(\K^{MW}_2)$, where the first morphism is the orientation character (associated with $\det \mathcal{E}$) and the second morphism is the standard $\gm$-action on $\Aut(\K^{MW}_2)$ coming from the action of $\bpi_1^{\aone}(BGL_2)$ on $\bpi_2^{\aone}(BGL_2)$ with orientation character corresponding to the universal bundle as in \ref{ex:universaleulerclass}.

Suppose the short exact sequence of the statement is split by a homomorphism $\bpi_1^{\aone}(X) \to \bpi_1^{\aone}({\mathbb P}({\mathcal E}))$, and consider the induced conjugation action of $\bpi_1^{\aone}(X)$ on the fiber; this action determines a homomorphism
\[
\bpi_1^{\aone}(X) \longrightarrow \Aut(\Faone(1))
\]
that completely determines the group structure.  Using the splitting, this conjugation action is defined in the same fashion as in topology.

Unwinding the definitions, we see that the conjugation action has geometric origin: we claim that it comes from the morphism ${\bf R}\Omega^1_s X \to PGL_2$ defining the connecting homomorphism.  Recall that the sequence of maps
\[
{\bf R}\Omega^1_s X \longrightarrow {\bf R}\Omega^1_s X\times \pone \longrightarrow PGL_2 \times \pone \longrightarrow \pone.
\]
Adjunction gives a morphism from ${\bf R}\Omega^1_s X$ to the space of self-maps of $\pone$, and this morphism factors through $PGL_2$.  Of course, the action of $PGL_2$ on $\pone$ moves base-points in the fiber.  Thus, the splitting shows that the morphism of sheaves $\bpi_1^{\aone}(X) \to \Aut(\Faone(1))$ factors through a morphism of sheaves
\[
\bpi_0^{\aone}(PGL_2) \to \Aut(\Faone(1))
\]
corresponding to the change of base-point.

This homomorphism is non-trivial since any element coming from the inclusion of the maximal torus $\gm \subset PGL_2$ fixes the base-point in $\pone$ and \cite[Theorem 7.35]{MField} shows that such homomorphisms can have non-trivial ``Brouwer degree."  Combining these two observations, we get a map
\[
\bpi_1^{\aone}(X) \to \Aut(\Faone(1)),
\]
that is completely determined by a homomorphism $\bpi_1^{\aone}(X) \to \bpi_0^{\aone}(PGL_2)$.  Tracing through the definitions and using Proposition \ref{prop:modnplus1orientationcharacter}, we see that the group structure on our projective bundle is uniquely determined by the class of $\det {\mathcal E}$ in $\Pic(X)/2\Pic(X)$.
\end{proof}

\begin{rem}
By Lemma \ref{lem:autfaone1}, we know that $\Aut(\Faone(1))$ is a sheaf of abelian groups.  As a consequence, any morphism $\bpi_1^{\aone}(X) \to \Aut(\Faone(1))$ factors through the first $\aone$-homology sheaf of $X$.  Using this observation, one can circumvent the discussion of base-points in the proposition.  Furthermore, there is another way to interpret the factorization through ${\mathcal H}^1_{\et}(\mu_2)$ described above.  Multiplication determines a morphism of sheaves $\K^{\MW}_0 \to \Aut(\K^{\MW}_2)$.  On the other hand, $\K^{\MW}_0$ is a quotient of the free strictly $\aone$-invariant sheaf of groups on the sheaf ${\mathcal H}^1_{\et}(\mu_2) = \gm/\gm^{\times 2}$.  Theorem \ref{thm:projectivebundletrivialeulerclass} thus asserts that the $\bpi_1^{\aone}(X)$ acts on the $\aone$-fundamental group of $\pone$ through the induced action on the $\aone$-fundamental group of ${\mathbb A}^2 \setminus 0$.
\end{rem}

\begin{ex}
The hypotheses of the Proposition are satisfied for $X = {\mathbb P}^n$, $n \geq 3$ and this gives part of Theorem \ref{thmintro:projectivebundles} from the introduction.  As a consequence, in this case, the obstruction to Hartshorne's conjecture stemming from the $\aone$-fundamental group is trivial.
\end{ex}

\begin{ex}
\label{ex:ponebundlesoverpone}
Again by Grothendieck's theorem \cite[Theorem 2.1.1]{OSS}, Theorem \ref{thm:projectivebundletrivialeulerclass} applies as well to projectivizations of rank $2$ vector bundles over $\pone$, in which case we recover \cite[Proposition 5.3.1]{AM}.
\end{ex}

\subsubsection*{Non-trivial Euler classes: an example}
Analogous to the topological situation, the Euler class provides an obstruction to splitting.  For bundles that do not split, the Euler class discussed above really can be non-trivial.  The next example is closely related to Remark \ref{rem:slnstabilization}.

\begin{ex}
\label{ex:nontrivialeulerclass}
Let $\mathbf{Fl}_3$ be the variety of complete flags in a $3$-dimensional $k$-vector space.  After fixing a flag, this space can be identified with the homogeneous space $SL_3/B$, where $B \subset SL_3$ is a Borel subgroup.  Proposition \ref{prop:fundamentalgroupproperties} demonstrates the existence of a short exact sequence of the form
\[
1 \longrightarrow \bpi_1^{\aone}(SL_3) \longrightarrow \bpi_1^{\aone}(SL_3/B) \longrightarrow T \to 1.
\]
Furthermore, we know that $\bpi_1^{\aone}(SL_3) = \K^M_2$ (and not $\K^{\MW}_2$) by Theorem \ref{thm:homotopygroupsofsln}.  By forgetting either the $1$ or $2$-dimensional subspace in a flag, we see that $SL_3/B$ fibers over ${\mathbb P}^2$ with fibers isomorphic to $\pone$.  Indeed, the morphism $SL_3/B \to {\mathbb P}^2$ is the projectivization of a tautological rank $2$ vector bundle over ${\mathbb P}^2$ (either the $2$-dimensional subspace of $V$ or the $2$-dimensional quotient depending on which projective space we choose), which is also indecomposable.

We know that ${\mathbb P}^2 = SL_3/P$, where $P$ is a parabolic subgroup of $SL_3$.  If $L$ is a Levi factor of $P$, then the induced map $SL_3/L \to SL_3/P$ by choice of a Levi splitting is Zariski locally trivial with affine space fibers.  Similarly, one deduces that the map $SL_3/T \to SL_3/B$ is an $\aone$-weak equivalence.  An inclusion of $T$ in $L$ the determines a morphism $SL_3/T \to SL_3/L$ that coincides, up to $\aone$-homotopy, with the above projective bundle.  Passing to $\aone$-connected $\aone$-covers one obtains a commutative diagram of the form
\[
\xymatrix{
SL_3 \ar[r]\ar[d] & SL_3/T \ar[d] \\
SL_3/SL_2 \ar[r] & SL_3/SL_2.
}
\]
The Euler class of the $SL_2$-bundle on the left is non-trivial(see \cite[Remark 7.18]{MField} or \cite[Lemma 3.1]{AsokFaselThreefolds}).  Moreover, one can show that this Euler class induces the (twisted) Euler class for the above projective bundle, which is therefore also non-trivial.
\end{ex}

\subsubsection*{Examples}
Combining Theorems \ref{thm:rankn}, \ref{thm:projectivebundletrivialeulerclass}, and Example \ref{ex:ranknexamples} we can give a complete computation of the $\aone$-fundamental group of a scroll (for brevity, we do not include the computations of Example \ref{ex:ponebundlesoverpone}, which are already discussed in \cite[Proposition 5.3.1]{AM}).

Fix an identification $\Z \isomt \Pic({\mathbb P}^m)$.  Given an $n$-tuple of integers ${\bf a} := (a_1,\ldots,a_n)$, consider the vector bundle $\O(a_1) \oplus \cdots \oplus \O(a_n)$ on ${\mathbb P}^m$.  Set $\ell({\bf a}) = m$ and call it the length of ${\mathbf{a}}$.

\begin{notation}
\label{notation:scrolls}
Set ${\mathbb F}_{m,{\bf a}} := {\mathbb P}_{{\mathbb P}^m}(\O(a_1) \oplus \cdots \oplus \O(a_n))$.  If $n = 2$, and ${\bf a} = (a,0)$, then set ${\mathbb F}_{m,a} := {\mathbb F}_{m,{\bf a}}$.
\end{notation}

\begin{thm}
\label{thm:scrolls}
Suppose $m$ and $n$ are integers $\geq 1$, and ${\bf a} = (a_1,\ldots,a_{n+1})$ is a sequence of integers. If $n \geq 2$, then there are isomorphisms
\[
\bpi_1^{\aone}({\mathbb F}_{m,{\bf a}}) \isomto \begin{cases} \gm \times \Faone(1) & \text{ if } m = 1 \\
\gm \times \gm & \text{ if } m > 1.
\end{cases}
\]
If $n = 1$ and $m > 1$, then there are isomorphisms
\[
\bpi_1^{\aone}({\mathbb F}_{m,a}) \isomto \begin{cases}
\Faone(1) \times \gm & \text{ if } a \cong 0 \mod 2 \\
\Faone(1) \rtimes \gm & \text{ if } a \cong 1 \mod 2,
\end{cases}
\]
where $a := a_1 + a_2$.
\end{thm}

\begin{ex}[$\aone$-fundamental groups of some blow-ups]
\label{ex:blowups}
One particular application of the results above is to computations of $\aone$-fundamental groups of blow-ups of linearly embedded projective subspaces of a given projective space.  Recall that the blow-up of a point $x$ in ${\mathbb P}^n$ ($n \geq 3$ for simplicity) is isomorphic to the projectivization of a direct sum of line bundles over ${\mathbb P}^{n-1}$.  The explicit description of these line bundles gives the isomorphism
\[
\bpi_1^{\aone}({\sf Bl}_{x}{\mathbb P}^n) \isomto \Faone(1) \rtimes \gm.
\]
More generally, if ${\mathbb P}^{n-k}$ is a linear subvariety of ${\mathbb P}^n$, with $k \geq 2$, then there is an $\aone$-weak equivalence ${\mathbb P}^{n} \setminus {\mathbb P}^{n-k} \to {\mathbb P}^{k-1}$ (in fact a vector bundle).  This morphism allows us to realize ${\sf Bl}_{{\mathbb P}^{n-k}} {\mathbb P}^n$ as the projectivization of a vector bundle over ${\mathbb P}^{k-1}$.  Thus, Theorem \ref{thm:scrolls} shows that
\[
\bpi_1^{\aone}({\sf Bl}_{{\mathbb P}^{n-k}} {\mathbb P}^n) \isomto \begin{cases} \Faone(1) \times \gm & \text{ if } k = 2\\
\gm \times \gm & \text{ if } 2 < k < n.
\end{cases}
\]
One could also use the $\aone$-van Kampen theorem \cite[Theorem 7.12]{MField} to approach these results, but the group structure on the extension in the cases where we blow-up a point or a codimension $2$ subvariety is not very transparent.
\end{ex}

\subsubsection*{Higher $\aone$-homotopy groups and higher obstructions}
We close with a very brief discussion of higher obstructions and their implications for the structure of higher $\aone$-homotopy groups of projectivizations of vector bundles over $\aone$-connected smooth schemes.  Above, we only really used the Euler class in the case of rank $2$ bundles.  When $X = {\mathbb P}^n$, $n \geq 3$, and $\mathcal{E}$ is a rank $2$ vector bundle on $X$, we observed that the Euler class $e({\mathcal E})$ of Definition \ref{defn:eulerclass} is trivial since $\bpi_2^{\aone}({\mathbb A}^{n+1} \setminus 0)$ vanishes in this situation.  Both of these observations factor into our construction of higher obstruction classes, which are only given for rather special situations.

If $X$ is an $\aone$-connected smooth proper scheme, then $\bpi_1^{\aone}(X)$ is always non-trivial (see Remark \ref{rem:pi1nontrivial}).  As a consequence, we cannot easily impose higher $\aone$-connectivity hypotheses on $X$.  Instead, we introduce the following property depending on an integer $i \geq 1$:
\begin{itemize}
\item[(${\mathbf C}_i$)] the universal $\aone$-covering space $\tilde{X}$ is $\aone$-$i$-connected.
\end{itemize}
Projective space ${\mathbb P}^n$ satisfies $({\mathbf C}_{n-1})$, and \cite[Theorem 6.4.2]{ADExcision} gives a condition on the fan guaranteeing that a smooth proper toric variety satisfies $({\mathbf C}_i)$ for some $i$.

Let $\mathcal{E}$ be a rank $r$ vector bundle on $X$.  If $f: \tilde{X} \to X$ is the $\aone$-universal covering morphism, then $f^*{\mathcal E}$ is determined by a simplicial homotopy class of maps $\tilde{X} \to BGL_r$.  Applying $\bpi_{i+1}^{\aone}(\cdot)$ to an explicit representative of this map gives a morphism $\bpi_{i+1}^{\aone}(\tilde{X}) \to \bpi_{i+1}^{\aone}(BGL_r)$.  If $X$ satisfies $({\mathbf C}_i)$, using \cite[Theorem 3.30]{ADExcision} we get a bijection
\[
H^{i+1}_{\Nis}(X,\bpi_{i+1}^{\aone}(BGL_r)) \isomto \hom_{\Ab^{\aone}_k}(\bpi_{i+1}^{\aone}(\tilde{X}),\bpi_{i+1}^{\aone}(BGL_r)).
\]
The element of $H^{i+1}_{\Nis}(X,\bpi_{i+1}^{\aone}(BGL_r))$ corresponding to the morphism $\bpi_{i+1}^{\aone}(\tilde{X}) \to \bpi_{i+1}^{\aone}(BGL_r)$ induced by the classifying map of $f^*\mathcal{E}$ will be denoted $e^i(\mathcal{E})$.  If $i+1 < r$, the $\aone$-homotopy groups of $BGL_r$ are already in the stable range (again, see Remark \ref{rem:slnstabilization}) and one can introduce algebraic K-theory into the discussion.

\begin{defn}
\label{defn:highereulerclass}
Keeping hypotheses as in the preceding paragraph, the class $e^i(\mathcal{E})$ will be called the $i$-ary Euler class of $\mathcal {E}$.
\end{defn}

\begin{rem}
This definition can be made more similar to the definition of the Euler class.  The vector bundle $\mathcal{E}$ is classified by a simplicial homotopy class of maps $X \to BGL_r$.  The homomorphism $\bpi_1^{\aone}(X) \to \bpi_1^{\aone}(BGL_r)$ therefore determines an action of $\bpi_1^{\aone}(X)$ on $\bpi_{i+1}^{\aone}(BGL_r)$ for $i \geq 0$.  Using this, one can define an ``$\aone$-local system" on $X$: it is the Nisnevich sheaf (on the small site of $X$) of local sections of the morphism $\tilde{X} \times_{\bpi_1^{\aone}(X)} \bpi_{i+1}^{\aone}(BGL_r) \to X$.  The $i$-ary Euler class of $X$ can then be viewed as a class on $X$ taking values in this sheaf.
\end{rem}

\begin{lem}
\label{lem:higherobstructions}
Suppose $X$ is an $\aone$-connected smooth scheme such that $\tilde{X}$ is $\aone$-$i$-connected for some integer $i \geq 2$.  Assume ${\mathcal E}$ is a rank $n+1$ vector bundle on $X$.  If $e^i({\mathcal E})$ is trivial, then connecting homomorphism $\bpi_{i+1}^{\aone}(X) \to \bpi_i^{\aone}({\mathbb P}^{n})$ in the long exact sequence of \textup{Corollary \ref{cor:projectivebundlelongexactsequence}} is trivial.  Therefore, there are induced isomorphisms \[
\bpi_i^{\aone}({\mathbb P}({\mathcal E})) \isomto \bpi_i^{\aone}({\mathbb P}^n).
\]
\end{lem}

\begin{proof}
The first statement is valid even for $i = 1$, in which case it is precisely Lemma \ref{lem:aonefundamentalgroupeulerclass}.  For $i \geq 1$, the proof is identical: one uses the observation that the map $\bpi_{i+1}^{\aone}(\tilde{X}) \to \bpi_{i+1}^{\aone}(X)$ is an isomorphism, and by construction, the connecting homomorphism factors through a map $X \to BSL_{n+1}$.  The second statement follows from the first using Lemma \ref{lem:aonecoveringsdontchangehigherhomotopy} and the $\aone$-connectivity assumption.
\end{proof}

\begin{ex}
The connectivity of $\tilde{X}$ will force vanishing of higher Euler classes.  For example, if $X = {\mathbb P}^n$, ($n \geq 3$), then $\bpi_i^{\aone}({\mathbb P}^n)$ vanishes for $1 < i < n$.  Given a rank $2$ vector bundle $\mathcal{E}$ on $X$, the first non-vanishing higher Euler class is
\[
e^n(\mathcal{E}) \in H^{n}_{\Nis}({\mathbb A}^{n+1} \setminus 0,\bpi_{n+1}^{\aone}(BGL_2)),
\]
and vanishing of this class completely controls the connecting homomorphism of Corollary \ref{cor:projectivebundlelongexactsequence} at the point $i = n$.
\end{ex}

\begin{rem}
In order to make the above results more effective, we need more information about the higher $\aone$-homotopy groups of $BGL_{n+1}$, or equivalently by Proposition \ref{prop:slnglnrelation} the higher $\aone$-homotopy groups of $SL_{n+1}$.  If $n \neq 1$, there is a nice description of these sheaves of groups in terms of so-called unstable Karoubi-Villamayor K-theory groups in \cite[Theorem 1]{WendtChevalley}.  The sheaf $\bpi_{n-1}^{\aone}(SL_n)$ for $n \geq 3$ (the case $n = 2$ was treated by Morel) has been completely determined in \cite{AsokFaselThreefolds} and \cite{AsokFaselSpheres}.  Also, the sheaf $\bpi_2^{\aone}(SL_2)$ was determined in \cite{AsokFaselThreefolds}.  These computations should allow one to make non-trivial statements about higher Euler classes.
\end{rem}

\section{Classification: $\aone$-homotopy groups and motivic cohomology rings}
\label{s:isomorphichomotopygroups}
In this section, we use the computations of $\aone$-homotopy groups from Section \ref{s:splittingobstructions} to produce examples of pairs of varieties that have isomorphic $\aone$-homotopy groups (of all degrees) yet that fail to be $\aone$-weakly equivalent.  For the sake of contrast, in Remark \ref{rem:lensspaces} we recall some facts about lens spaces, which provide either examples of homotopy inequivalent varieties with isomorphic homotopy groups and isomorphic cohomology rings.  The main result of this section is Theorem \ref{thm:main}.  In the examples, the $\aone$-fundamental group is the hardest $\aone$-homotopy group to control, and as a consequence dimension $3$ is the most difficult dimension in which to construct all the relevant examples.

\begin{rem}
\label{rem:lensspaces}
For the purposes of motivation, we recall that examples of weak-homotopy inequivalent $3$-manifolds with abstractly isomorphic homotopy groups and cohomology rings go back to the birth of geometric topology.  Indeed, fix a prime $p$, an integer $q$ coprime to $p$, and let $\zeta = e^{\frac{2\pi i}p}$.  Choose coordinates $(x_1,x_2)$ on $\cplx^2$, and consider the action of the cyclic group $\Z/p$ generated by $\zeta \cdot (x_1,x_2) = (\zeta x_1, \zeta^q x_2)$.  The resulting action is free and restricts to an action of $\Z/p$ on $S^3 \subset \cplx^2$; the quotient of $S^3$ by this free $\Z/p$-action is the $3$-dimensional lens space $L(p,q)$.  The homotopy classification of $3$-dimensional lens spaces goes back to Whitehead \cite[p. 1198]{Whitehead}: $L(p,q)$ is homotopy equivalent to $L(p,q')$ if and only if $qq'$ is a square mod $p$. Covering space theory shows that the homotopy groups of $L(p,q)$ are independent of $q$.  Using the classical Hurewicz and the universal coefficient theorems, one sees that the cohomology ring of $L(p,q)$ is independent of $q$ (the cup product is trivial).
\end{rem}

\subsubsection*{$\aone$-weak equivalences of scrolls}
We will need some rather explicit information about the construction of $\aone$-weak equivalences of scrolls.  Recall from Notation \ref{notation:scrolls}, that if ${\mathbf a} = (a_1,\ldots,a_n)$, then ${\mathbb F}_{m,{\mathbf a}} = {\mathbb P}_{{\mathbb P}^m}(\O(a_1) \oplus \cdots \oplus \O(a_n))$.  When $n = 2$, up to twisting by a line bundle, all scrolls are of the form ${\mathbb F}_{{m,(0,a)}}$ for some integer $a$, and we will abbreviate this to ${\mathbb F}_{m,a}$.  We begin by recalling the following result.

\begin{prop}[{\cite[Proposition 3.2.10]{AM}}]
\label{prop:hirzebruchsurfaces}
Fix $n+1$-tuples of integers ${\bf a} = (a_1,\ldots,a_{n+1})$ and ${\bf a}' = (a_1',\ldots,a_{n+1}')$.  The scrolls ${\mathbb F}_{1,{\bf a}}$ and ${\mathbb F}_{1,{\bf a}'}$ are $\aone$-weakly equivalent if and only if $\sum_i a_i = \sum_i a_{i}' \mod n+1$ .
\end{prop}

The $\aone$-weak equivalences in the proposition are constructed by means of explicit $\aone$-$h$-cobordisms.  To distinguish the varieties in question, one can use the computation of Chow cohomology rings; we review some notation because it will be used in our computations below.  Suppose $\mathcal{E}$ is a rank $r$ vector bundle over a projective space ${\mathbb P}^n$.  The Chow cohomology ring of the projective bundle $\mathbb{P}_{\mathbb{P}^n}({\mathcal E})$ is computed as follows.  Let $\xi \in H^{2,1}({\mathbb P}^n,\Z)$ be the first Chern class of $\O(1)$.  Let $P_\tau(\mathcal{E})$ be the Chern polynomial of $\mathcal{E}$ defined by
\[
P_{\tau}(\mathcal{E}) = \sum_i \tau^{d - i}c_i(\mathcal{E}).
\]
Then,
\[
H^{*,*}({\mathbb P}_{\mathbb{P}^n}(\mathcal{E}),\Z) \isomto H^{*,*}(\Spec k)[\xi,\tau]/\langle \xi^{n+1},P_{\tau}(\mathcal{E}) \rangle;
\]
here $\tau$ also has bidegree $(2,1)$ and we use the notation of motivic cohomology.  One can identify the subring $\oplus_{i \geq 0} H^{2i,i}(X,\Z)$ of motivic cohomology with the Chow ring $CH^*(X)$; see \cite[Corollary 19.2]{MVW}.

\begin{ex}
\label{ex:smallscrolls}
If ${\bf a} = (a_1,a_2)$, the Chow cohomology ring of the scroll ${\mathbb F}_{1,{\bf a}}$ takes the form $\Z[\xi,\tau]/\langle \xi^2,\tau^3 + (a_1+a_2)\xi \tau^2 \rangle$. Similarly, the Chow cohomology ring of the scroll ${\mathbb F}_{2,b}$ takes the form $\Z[\xi,\tau]/\langle \xi^3,\tau^2 + b\xi \tau \rangle$.  Any isomorphism of graded rings is induced by a $GL_2(\Z)$ action on the generators $\xi,\tau$.
\end{ex}

\subsubsection*{$\aone$-weakly inequivalent varieties with isomorphic $\aone$-homotopy groups}
\begin{thm}
\label{thm:main}
Let $n \geq 2$ be an integer, and suppose ${\bf a} = (a_1,a_2,a_3)$ is a sequence of integers.  For every $i \geq 1$ there are isomorphisms
\[
\bpi_i^{\aone}({\mathbb F}_{1,{\bf a}}) \isomto \bpi_i^{\aone}({\mathbb P}^2) \times \bpi_i^{\aone}({\mathbb P}^1).
\]
However, if ${\bf a}' := (a_1',a_2',a_3')$, then ${\mathbb F}_{1,{\bf a}}$ is $\aone$-weakly equivalent to ${\mathbb F}_{1,{\bf a}'}$ if and only if $\sum_i a_i = \sum_i a_i' \mod 3$.
\end{thm}

\begin{thm}
\label{thm:main2}
If $a$ and $a'$ are integers that are congruent modulo $2$, then for every integer $i \geq 0$ there are isomorphisms
\[
\bpi_i^{\aone}({\mathbb F}_{2,a}) \cong \bpi_i^{\aone}({\mathbb F}_{2,a'}),
\]
but ${\mathbb F}_{2,a}$ and ${\mathbb F}_{2,b}$ are $\aone$-weakly equivalent if and only if $a = a'$.
\end{thm}

\begin{proof}[Proofs of Theorems \ref{thm:main} and \ref{thm:main2}]
In both cases, the isomorphism of $\aone$-fundamental groups follows from Theorem \ref{thm:scrolls}, and the isomorphisms between higher $\aone$-homotopy groups are an immediate consequence of Proposition \ref{prop:splitvectorbundles}.

{\em Case of Theorem \ref{thm:main}}.  That the varieties in question are $\aone$-weakly equivalent under the stated hypothesis is \cite[Proposition 3.2.10]{AM}.  Consider the binary cubic form on $Pic(X)$ induced by the ring structure.  Setting $b = a_1 + a_2 + a_3$ and choosing coordinates $x_1$ and $x_2$ dual to $\xi$ and $\tau$, explicit computation using the discussion of Example \ref{ex:smallscrolls} identifies this cubic form as $(3x_1 - bx_2)x_2^2$.  If $b' = a_1' + a_2' + a_3'$, a straightforward computation of the $GL_2(\Z)$-equivalence class of this binary cubic form shows that it is equivalent to $(3x_1 - b'x_2)x_2^2$ if and only if $b$ and $b'$ are congruent modulo $3$.

{\em Case of Theorem \ref{thm:main2}}. Again, using the discussion of Example \ref{ex:smallscrolls} one sees that the discriminant of the binary cubic form on $Pic(X)$ induced by the ring structure is $27a^2 + 108a^5$.  If $a \neq a'$, then it is not hard to show that there are no non-trivial integral solutions to the equation $27a^2 + 108 a^5 = 27 (a')^2 + 108 a'^5$ (any non-trivial real solution to this equation has $(a,a')$ contained in the open box $(-1,1) \times (-1,1)$).
\end{proof}

\begin{rem}
\label{rem:higherdimensionalexamplesI}
Producing examples of pairs of $\aone$-connected smooth proper varieties that have isomorphic $\aone$-homotopy groups yet which are $\aone$-weakly inequivalent in dimension $\geq 4$ is significantly easier than in dimension $3$.  For example, suppose $m$ and $n$ are integers $\geq 2$.  It is easy to check that any two scrolls ${\mathbb F}_{m,{\bf a}}$ and ${\mathbb F}_{n,{\bf b}}$ with $m + \ell({\bf a}) = n + \ell({\bf b}) \geq 4$ have isomorphic $\aone$-homotopy groups for all $i > 0$.  Indeed, the $\aone$-universal covers of ${\mathbb F}_{n,{\bf b}}$ and ${\mathbb F}_{m,{\bf a}}$ are both isomorphic to ${\mathbb A}^{m+1} \setminus 0 \times {\mathbb A}^{n+1} \setminus 0$.  In both cases, the $\aone$-fundamental group is isomorphic to $\gm \times \gm$, and the isomorphism of the previous sentence, combined with Proposition \ref{prop:fundamentalgroupproperties}.iii, yields isomorphisms of higher $\aone$-homotopy sheaves.  Nevertheless, one can produce examples of such scrolls with non-isomorphic cohomology rings.
\end{rem}

\begin{rem}
The isomorphism on $\aone$-homotopy groups is not, in general, induced by a morphism of spaces, even in the $\aone$-homotopy category; by the $\aone$-Whitehead theorem \cite[\S 3 Proposition 2.14]{MV}, it can be induced by a morphism of spaces if and only if the underlying spaces are $\aone$-weakly equivalent.
\end{rem}

\begin{rem}
Blowing up the same number of points on each of ${\mathbb F}_{1,{\bf a}}$ and $\mathbb{F}_{2,b}$ produces many more examples of varieties with isomorphic $\aone$-homotopy groups.  However, $\aone$-weakly inequivalent varieties can become $\aone$-weakly equivalent after blowing-up: recall that ${\mathbb F}_1$ is not $\aone$-weakly equivalent to ${\mathbb F}_2$, but after blowing-up an appropriately chosen point on each, they become isomorphic.

More generally, all of the examples constructed above are non-singular projective toric varieties.  Given two toric projective bundles $X$ and $X'$ it is not difficult to give a non-singular projective toric variety $X''$ birationally equivalent to and dominating both $X$ and $X'$ which is constructed from either of $X$ or $X'$ by means of an explicit sequence of blow-ups at non-singular toric subvarieties.  In other words, all of the examples constructed above, while $\aone$-weakly inequivalent, become $\aone$-weakly equivalent after blowing-up.
\end{rem}

\begin{rem}
\label{rem:3dclassification}
In the study of lens spaces (and in fact for closed oriented $3$-manifolds in general), the homotopy classification problem can be attacked by means of obstruction theory (see, e.g., \cite[Theorems V and VI]{Olum} and the references in that paper, or \cite{Thomas,Swarup} for a more geometric point of view).  Indeed, to perform this classification two invariants are necessary: the (usual) fundamental group and an appropriate notion of ``twisted degree" (see also \cite[Theorem I]{Franz}) involving the orientation character.  In the context of $\aone$-homotopy theory of projective bundles, one can also introduce various notions of degree.  However, in addition to keeping track of the $\aone$-fundamental sheaf of groups, the above results suggest that in order to perform the $\aone$-homotopy classification it is also necessary to keep track of (at least) the cubic form on $Pic(X)$ defined by the intersection pairing (see \cite{OVdV}). The $\aone$-fundamental group controls $Pic(X)$ by means of Proposition \ref{prop:fundamentalgroupproperties}(i), but in order to keep track of the above cubic form, one needs a fundamental class.
\end{rem}

\begin{ex}
The $\aone$-connected smooth proper $3$-folds with $\aone$-fundamental group $\Faone(1) \times \gm$ include
all ${\mathbb P}^2$-bundles over $\pone$, and projectivizations of rank $2$ vector bundles $\mathcal{E}$ on ${\mathbb P}^2$ with even first Chern class and trivial Euler class.  There are non-trivial $\aone$-weak equivalences among varieties of the first type, but it is not clear whether there are non-trivial weak equivalences among varieties of the second type.
\end{ex}

\begin{rem}
There is another point of view on the above examples: all the varieties for which we have produced isomorphic $\aone$-homotopy groups arise as $\gm \times \gm$-quotients of products of pairs of punctured affine spaces.  The punctured affine spaces arise as stable points for different linearizations of $\gm \times \gm$-actions on affine space.  All the projectivizations of split vector bundles above are realized in this fashion by changing linearizations for fixed actions on a given ambient affine space.  The $\aone$-fundamental groups of the quotients can be described by means of Proposition \ref{prop:fundamentalgroupproperties}, and the group structure on the $\aone$-fundamental group, i.e., the extension in question, is completely determined by the linearization, which is determined by a character of $\gm \times \gm$.  The space of characters of $\gm \times \gm$ breaks into finitely many chambers such that all the quotients for linearizations in a fixed chamber are isomorphic.  When passing through a wall, the birational type of the variety in question remains the same, while the isomorphism class can change; variation of GIT in this particular instance is worked out in considerable detail in \cite[Appendix A]{BCZ}.  When the resulting birational transformations are sufficiently ``small," one can use geometric arguments to show that $\aone$-fundamental groups do not change.
\end{rem}

\begin{footnotesize}
\bibliographystyle{alpha}
\bibliography{threefolds}

\begin{thebibliography}{OVdV95}

\bibitem[AD09]{ADExcision}
A.~Asok and B.~Doran.
\newblock {${\mathbb A}^1$}-homotopy groups, excision, and solvable quotients.
\newblock {\em Adv. Math.}, 221(4):1144--1190, 2009.

\bibitem[AF12a]{AsokFaselSpheres}
A.~Asok and J.~Fasel.
\newblock Algebraic vector bundles on spheres.
\newblock {\em Preprint} available at \url{http://arxiv.org/abs/1204.4538},
  2012.

\bibitem[AF12b]{AsokFaselThreefolds}
A.~Asok and J.~Fasel.
\newblock A cohomological classification of vector bundles on smooth affine
  threefolds.
\newblock {\em Preprint} available at \url{http://arxiv.org/abs/1204.0770},
  2012.

\bibitem[AF12c]{AsokFaselpi3a3minus0}
A.~Asok and J.~Fasel.
\newblock Splitting vector bundles outside the stable range and homotopy theory
  of punctured affine spaces.
\newblock {\em In preparation}, 2012.

\bibitem[AM11]{AM}
A.~Asok and F.~Morel.
\newblock Smooth varieties up to {${\mathbb A}^1$}-homotopy and algebraic
  {$h$}-cobordisms.
\newblock {\em Adv. Math.}, 227(5):1990--2058, 2011.

\bibitem[BCZ04]{BCZ}
G.~Brown, A.~Corti, and F.~Zucconi.
\newblock Birational geometry of 3-fold {M}ori fibre spaces.
\newblock In {\em The {F}ano {C}onference}, pages 235--275. Univ. Torino,
  Turin, 2004.

\bibitem[BM00]{BargeMorel}
J.~Barge and F.~Morel.
\newblock Groupe de {C}how des cycles orient\'es et classe d'{E}uler des
  fibr\'es vectoriels.
\newblock {\em C. R. Acad. Sci. Paris S\'er. I Math.}, 330(4):287--290, 2000.

\bibitem[D{\'e}g11]{DegliseOriented}
F.~D{\'e}glise.
\newblock Oriented homotopy modules.
\newblock {\em {\em To appear} Amer. j. Math.}, 2011.

\bibitem[Fas08]{Fasel1}
J.~Fasel.
\newblock Groupes de {C}how-{W}itt.
\newblock {\em M\'em. Soc. Math. Fr. (N.S.)}, (113):viii+197, 2008.

\bibitem[Fra43]{Franz}
W.~Franz.
\newblock Abbildungsklassen und {F}ixpunktklassen dreidimensionaler
  {L}insenr\"aume.
\newblock {\em J. Reine Angew. Math.}, 185:65--77, 1943.

\bibitem[GJ99]{GoerssJardine}
P.~G. Goerss and J.~F. Jardine.
\newblock {\em Simplicial homotopy theory}, volume 174 of {\em Progress in
  Mathematics}.
\newblock Birkh\"auser Verlag, Basel, 1999.

\bibitem[Gro61]{EGAII}
A.~Grothendieck.
\newblock \'{E}l\'ements de g\'eom\'etrie alg\'ebrique. {II}. \'{E}tude globale
  \'el\'ementaire de quelques classes de morphismes.
\newblock {\em Inst. Hautes \'Etudes Sci. Publ. Math.}, (8):222, 1961.

\bibitem[Gro68]{BrauerIII}
A.~Grothendieck.
\newblock Le groupe de {B}rauer. {III}. {E}xemples et compl\'ements.
\newblock In {\em Dix {E}xpos\'es sur la {C}ohomologie des {S}ch\'emas}, pages
  88--188. North-Holland, Amsterdam, 1968.

\bibitem[Har74]{Hartshorne}
R.~Hartshorne.
\newblock Varieties of small codimension in projective space.
\newblock {\em Bull. Amer. Math. Soc.}, 80:1017--1032, 1974.

\bibitem[Hen77]{Hendriks}
H.~Hendriks.
\newblock Applications de la th\'eorie d'obstruction en dimension {$3$}.
\newblock {\em Bull. Soc. Math. France M\'em.}, (53):81--196, 1977.

\bibitem[Hov99]{Hovey}
M.~Hovey.
\newblock {\em Model categories}, volume~63 of {\em Mathematical Surveys and
  Monographs}.
\newblock American Mathematical Society, Providence, RI, 1999.

\bibitem[MFK94]{GIT}
D.~Mumford, J.~Fogarty, and F.~Kirwan.
\newblock {\em Geometric invariant theory}, volume~34 of {\em Ergebnisse der
  Mathematik und ihrer Grenzgebiete (2) [Results in Mathematics and Related
  Areas (2)]}.
\newblock Springer-Verlag, Berlin, third edition, 1994.

\bibitem[Mor04]{MIntro}
F.~Morel.
\newblock An introduction to {${\mathbb A}^1$}-homotopy theory.
\newblock In {\em Contemporary developments in algebraic $K$-theory}, ICTP
  Lect. Notes, XV, pages 357--441 (electronic). Abdus Salam Int. Cent. Theoret.
  Phys., Trieste, 2004.

\bibitem[Mor05]{MStable}
F.~Morel.
\newblock The stable {${\mathbb A}^1$}-connectivity theorems.
\newblock {\em $K$-Theory}, 35(1-2):1--68, 2005.

\bibitem[Mor06]{MICM}
F.~Morel.
\newblock {${\mathbb A}^1$}-algebraic topology.
\newblock In {\em International Congress of Mathematicians. Vol. II}, pages
  1035--1059. Eur. Math. Soc., Z\"urich, 2006.

\bibitem[Mor11]{MFM}
F.~Morel.
\newblock On the {F}riedlander-{M}ilnor conjecture for groups of small rank.
\newblock In {\em Current developments in mathematics, 2010}, pages 45--93.
  Int. Press, Somerville, MA, 2011.

\bibitem[Mor12]{MField}
F.~Morel.
\newblock {\em {${\mathbb A}^1$}-algebraic topology over a field}, volume 2052
  of {\em Lecture Notes in Mathematics}.
\newblock Springer, Heidelberg, 2012.

\bibitem[MS74]{MilnorStasheff}
J.~W. Milnor and J.~D. Stasheff.
\newblock {\em Characteristic classes}.
\newblock Princeton University Press, Princeton, N. J., 1974.
\newblock Annals of Mathematics Studies, No. 76.

\bibitem[MV99]{MV}
F.~Morel and V.~Voevodsky.
\newblock {${\mathbb A}^1$}-homotopy theory of schemes.
\newblock {\em Inst. Hautes \'Etudes Sci. Publ. Math.}, (90):45--143 (2001),
  1999.

\bibitem[MVW06]{MVW}
C.~Mazza, V.~Voevodsky, and C.~Weibel.
\newblock {\em Lecture notes on motivic cohomology}, volume~2 of {\em Clay
  Mathematics Monographs}.
\newblock American Mathematical Society, Providence, RI, 2006.

\bibitem[Olu53]{Olum}
P.~Olum.
\newblock Mappings of manifolds and the notion of degree.
\newblock {\em Ann. of Math. (2)}, 58:458--480, 1953.

\bibitem[OSS80]{OSS}
C.~Okonek, M.~Schneider, and H.~Spindler.
\newblock {\em Vector bundles on complex projective spaces}, volume~3 of {\em
  Progress in Mathematics}.
\newblock Birkh\"auser Boston, Mass., 1980.

\bibitem[OVdV95]{OVdV}
Ch. Okonek and A.~Van~de Ven.
\newblock Cubic forms and complex {$3$}-folds.
\newblock {\em Enseign. Math. (2)}, 41(3-4):297--333, 1995.

\bibitem[Swa74]{Swarup}
G.~A. Swarup.
\newblock On a theorem of {C}. {B}. {T}homas.
\newblock {\em J. London Math. Soc. (2)}, 8:13--21, 1974.

\bibitem[Tho69]{Thomas}
C.~B. Thomas.
\newblock The oriented homotopy type of compact {$3$}-manifolds.
\newblock {\em Proc. London Math. Soc. (3)}, 19:31--44, 1969.

\bibitem[Tot99]{Totaro}
B.~Totaro.
\newblock The {C}how ring of a classifying space.
\newblock In {\em Algebraic {$K$}-theory ({S}eattle, {WA}, 1997)}, volume~67 of
  {\em Proc. Sympos. Pure Math.}, pages 249--281. Amer. Math. Soc., Providence,
  RI, 1999.

\bibitem[Wen10]{WendtChevalley}
M.~Wendt.
\newblock {${\mathbb A}^1$}-homotopy of {C}hevalley groups.
\newblock {\em J. K-Theory}, 5(2):245--287, 2010.

\bibitem[Wen11]{Wendt}
M.~Wendt.
\newblock Rationally trivial torsors in {${\mathbb A}^1$}-homotopy theory.
\newblock {\em J. K-Theory}, 7(3):541--572, 2011.

\bibitem[Whi41]{Whitehead}
J.~H.~C. Whitehead.
\newblock On incidence matrices, nuclei and homotopy types.
\newblock {\em Ann. of Math. (2)}, 42:1197--1239, 1941.

\end{thebibliography}
\end{footnotesize}
\end{document}